\documentclass[12pt]{article}

\usepackage[utf8]{inputenc}
\usepackage[english]{babel}
\usepackage{amsthm,amsmath,mathtools,amsfonts,amssymb}

\usepackage{authblk}

\usepackage{geometry}
\usepackage{blindtext}
\usepackage{fancyhdr}
\usepackage{caption}
\usepackage{subcaption}

\usepackage{graphicx,float}
\usepackage{enumerate}

\usepackage{bbm}

\numberwithin{equation}{section}

\usepackage{dirtytalk}

\usepackage[comma,authoryear,round]{natbib}

\usepackage{url}

\usepackage[colorlinks,linkcolor=magenta,citecolor=blue,pagebackref=true]{hyperref}

\renewcommand*{\backrefalt}[4]{%
	\ifcase #1 \footnotesize{(Not cited.)}%
	\or        \footnotesize{(Cited on page~#2.)}%
	\else      \footnotesize{(Cited on pages~#2.)}%
	\fi}

\usepackage[nameinlink,capitalize,noabbrev]{cleveref}

\usepackage{nicefrac}

\providecommand{\keywords}[1]
{
	\small	
	\textbf{\textit{Keywords.}} #1
}

\newcounter{algsubstate}

\def\ie{{\em i.e.,~}}
\def\st{{\em s.t.~}}
\def\eg{{\em e.g.,~}}

\newcommand{\wrt}{w.r.t.}

\DeclareMathOperator*{\argmin}{argmin}

\newcommand{\zero}{\ensuremath{\mathbf{0}}}

\DeclareMathOperator*{\vect}{vec} 
\DeclareMathOperator*{\kl}{KL}
\DeclareMathOperator*{\tkl}{KL^{\otimes n}}
\DeclareMathOperator*{\jtkl}{JKL_\rho^{\otimes n}}

\DeclareMathOperator*{\hel}{d^2}
\DeclareMathOperator*{\thel}{d^{2 \otimes n}}
\DeclareMathOperator*{\thell}{d^{\otimes n}}
\DeclareMathOperator*{\entropy}{\mathcal{H}_{d_{\norm{\sup}_\infty}}}

\DeclareMathOperator*{\card}{card}

\DeclareMathOperator*{\disc}{disc} 

\newcommand{\sbm}{\cS^{\cB}_{(K,J)}} 

\let\inf\relax 
\DeclareMathOperator*\inf{\vphantom{p}inf}

\DeclareMathOperator*{\adj}{Adj}

\DeclarePairedDelimiter\norm{\lVert}{\rVert}

\newcommand{\gkdb}{\cG_{\left(K,d,\bfB\right)}} 
\newcommand{\gdbk}{\cG_{\left(d,\bfB_k\right)}} 

\newcommand{\UpsilondK}{\Upsilonb_{\left(K,d\right)}} 

\newcommand{\Upsilondk}{\Upsilonb_{\left(k,d\right)}} 

\newcommand{\nn}{\nonumber} 
\newcommand{\R}{\mathbb{R}} 

\newcommand{\Indi}{\mathbb{I}} 
\newcommand{\Ns}{\mathbb{N}^\star} 
\newcommand{\pen}{\text{pen}} 
\newcommand{\Z}{\mathbb{Z}} 

\newcommand{\fC}{\mathfrak{C}}

\newcommand{\E}[2]{\mathbb{E}_{#1} \left[#2\right]}

\newcommand{\cA}{\mathcal{A}}
\newcommand{\cb}{\mathcal{b}}
\newcommand{\cB}{\mathcal{B}}

\newcommand{\cD}{\mathcal{D}}

\newcommand{\cG}{\mathcal{G}}

\newcommand{\cH}{\mathcal{H}}

\newcommand{\cK}{\mathcal{K}}

\newcommand{\cM}{\mathcal{M}}

\newcommand{\cN}{\mathcal{N}}

\newcommand{\cP}{\mathcal{P}}

\newcommand{\cR}{\mathcal{R}}

\newcommand{\cS}{\mathcal{S}}

\newcommand{\cU}{\mathcal{U}}

\newcommand{\cW}{\mathcal{W}}

\newcommand{\cX}{\mathcal{X}}

\newcommand{\cY}{\mathcal{Y}}

\newcommand{\bfA}{\mathbf{A}}
\newcommand{\bfb}{\mathbf{b}}
\newcommand{\bfB}{\mathbf{B}}
\newcommand{\bfc}{\mathbf{c}}
\newcommand{\bfC}{\mathbf{C}}

\newcommand{\bfE}{\mathbf{E}}

\newcommand{\bfg}{\mathbf{g}}

\newcommand{\bfm}{\mathbf{m}}

\newcommand{\bfP}{\mathbf{P}}

\newcommand{\bfr}{\mathbf{r}}

\newcommand{\bfs}{\mathbf{s}}

\newcommand{\bft}{\mathbf{t}}

\newcommand{\bfV}{\mathbf{V}}
\newcommand{\bfw}{\mathbf{w}}

\newcommand{\bfx}{\mathbf{x}}

\def\alphab      {{\sbmm{\alpha}}\XS}

\def\Gammab    {{\sbmm{\Gamma}}\XS}

\def\thetab      {{\sbmm{\theta}}\XS}      \def\Thetab    {{\sbmm{\Theta}}\XS}

\def\mub         {{\sbmm{\mu}}\XS}

\def\pib         {{\sbmm{\pi}}\XS}                 \def\Pib        {{\sbmm{\Pi}}\XS}

      \def\Sigmab    {{\sbmm{\Sigma}}\XS}

\def\psib        {{\sbmm{\psi}}\XS}        \def\Psib       {{\sbmm{\Psi}}\XS}
\def\omegab      {{\sbmm{\omega}}\XS}      \def\Omegab    {{\sbmm{\Omega}}\XS}

\def\upsilonb    {{\sbmm{\upsilon}}\XS}   \def\Upsilonb  {{\sbmm{\Upsilon}}\XS}

\def\upsilonb       {{\sbmm{\upsilon}}\XS}

\newcommand{\Xv}{\ensuremath{\mathbf{X}}} 
\newcommand{\xv}{\ensuremath{\mathbf{x}}} 
\newcommand{\Yv}{\ensuremath{\mathbf{Y}}}
\newcommand{\yv}{\ensuremath{\mathbf{y}}}

\def\XS{\xspace}
\DeclareMathAlphabet{\mathb}{OML}{cmm}{b}{it}
\def\sbm#1{\ensuremath{\mathb{#1}}}
\def\sbmm#1{\ensuremath{\boldsymbol{#1}}}  

\def\Ab{{\sbm{A}}\XS}  
  \def\bb{{\sbm{b}}\XS}
  \def\cb{{\sbm{c}}\XS}

\def\Vb{{\sbm{V}}\XS}

\theoremstyle{plain}
\newtheorem{theorem}{Theorem}[section]
\newtheorem{lem}[theorem]{Lemma}

\newtheorem{prop}[theorem]{Proposition}

\newtheorem{assumption}{Assumption}[section]
\theoremstyle{definition}

\newtheorem{remark}[theorem]{Remark}

\crefname{lem}{Lemma}{Lemmas}
\crefname{theorem}{Theorem}{Theorems}
\crefname{prop}{Proposition}{Propositions}
\theoremstyle{plain}

\begin{document}


\title{Non-asymptotic model selection in \\block-diagonal mixture of polynomial experts models}

\author[1]{TrungTin Nguyen\thanks{Corresponding author, email: trung-tin.nguyen@unicaen.fr}}
\author[1]{Faicel Chamroukhi}
\author[2]{Hien Duy Nguyen}
\author[3]{Florence Forbes}

\affil[1]{Normandie Univ, UNICAEN, CNRS, LMNO, 14000 Caen, France.}
\affil[2]{Department of Mathematics and Statistics, La Trobe University, Bundoora Melbourne 3066, Victoria Australia.}
\affil[3]{Univ. Grenoble Alpes, Inria, CNRS, Grenoble INP, LJK, Inria Grenoble Rhone-Alpes, 655 av. de l’Europe, 38335 Montbonnot, France.}


\maketitle

\begin{abstract}
	Model selection, via penalized likelihood type criteria, is a standard task in many statistical inference and machine learning problems. Progress has led to deriving criteria with asymptotic consistency results and an increasing emphasis on introducing non-asymptotic criteria. 
	We focus on the problem of modeling non-linear relationships in regression data with potential hidden graph-structured interactions between the high-dimensional predictors, within the mixture of experts modeling framework. 
	In order to deal with such a complex situation, we investigate a block-diagonal localized mixture of polynomial experts (BLoMPE) regression model, which is constructed upon an inverse regression and block-diagonal structures of the Gaussian expert covariance matrices. 
	We introduce a penalized maximum likelihood selection criterion to estimate the unknown conditional density of the regression model. 
	This model selection criterion allows us to handle the challenging problem of inferring the number of mixture components, the degree of polynomial mean functions, and the hidden block-diagonal structures of the covariance matrices, which reduces the number of parameters to be estimated and leads to a trade-off between complexity and sparsity in the model.
	In particular, we provide a strong theoretical guarantee: a finite-sample oracle inequality satisfied by the penalized maximum likelihood estimator with a Jensen--Kullback--Leibler type loss, to support the introduced non-asymptotic model selection criterion. 
	The penalty shape of this criterion depends on the complexity of the considered random subcollection of BLoMPE models, including the relevant graph structures, the degree of polynomial mean functions, and the number of mixture components.
\end{abstract}


\pagebreak
\tableofcontents

\section{Introduction}
Mixture of experts (MoE) models, initially proposed in \cite{jacobs1991adaptive,jordan1994hierarchical}, have been thoroughly studied in statistics and machine learning due to their flexibility and the abundance of statistical estimation and model selection tools available to fit them. Their universal approximation capability has been extensively established for not only finite mixture models \citep{genovese2000rates, nguyen2013convergence, ho2016convergence, ho2016strong, nguyen2020approximation,nguyen2020approximationLesbegue} but also conditional densities of MoE  \citep{jiang1999hierarchical, norets2010approximation, mendes2012convergence, nguyen2016universal, ho2019convergence,nguyen2019approximation,nguyen2020approximationMoE}. Recent reviews on practical and theoretical issues of MoE models can be found in \cite{yuksel2012twenty}, and \cite{nguyen2018practical}.

In this paper, we aim to provide a non-asymptotic oracle type inequality, which is well-known as a strong theoretical guarantee for the slope heuristic procedure for model selection \citep{birge2007minimal}, in block-diagonal localized mixture of polynomial experts (BLoMPE) models. 
The BLoMPE model is a Gaussian-gated localized mixture of experts (GLoME) \cite{xu1995alternative} enjoying a parsimonious covariance structure,
via block-diagonal structures for covariance matrices in the Gaussian experts. 
It is worth mentioning that GLoME models have been utilized extensively under several different contexts in statistics and machine learning: localized MoE \citep{ramamurti1996structural,ramamurti1998use,moerland1999classification,bouchard2003localised}, normalized Gaussian network \citep{sato2000line}, MoE modeling of priors in Bayesian nonparametric regression \citep{norets2014posterior,norets2017adaptive},
cluster-weighted modeling \citep{ingrassia2012local}, supervised Gaussian locally-linear mapping (GLLiM) in inverse regression \citep{deleforge2015high}, deep mixture of linear inverse regressions \citep{Lathuiliere2017},  multiple-output Gaussian gated
mixture of linear experts \citep{nguyen2019approximation} and regularized Gaussian gated
mixture of experts \citep{chamroukhi2019regularized}, to name just a few.

It is also interesting to point out that block-diagonal covariance for Gaussian locally-linear mapping (BLLiM) model in \cite{devijver2017nonlinear} is an affine instance of a BLoMPE model, where linear combination of bounded functions (\eg polynomials) are considered instead of affine mean functions for the Gaussian experts. 
The BLLiM framework aims to model a sample of high-dimensional regression data issued from a heterogeneous population with hidden graph-structured interaction between covariates. In particular, the BLLiM model is considered as a good candidate for performing a model-based clustering and predicting the response in situations affected by the curse of dimensionality phenomenon, where the number of parameters could be larger than the sample size. Indeed, to deal with high-dimensional regression problems, the BLLiM model, initially proposed by \cite{li1991sliced}, is based on an inverse regression strategy, which inverts the role of the high-dimensional predictor and the multivariate response. Therefore, the number of parameters to estimate is drastically reduced. 
More precisely, BLLiM utilizes the Gaussian locally-linear mapping (GLLiM), described in \cite{deleforge2015hyper,deleforge2015high}, and \cite{perthame2018inverse}, in conjunction with a block-diagonal structure hypothesis on the residual covariance matrices  to make a trade-off between complexity and sparsity.

This prediction model is fully parametric and highly interpretable. For instance, it might be useful for
the analysis of transcriptomic data in molecular biology to classify observations or predict phenotypic states, as for example disease versus non disease or tumor versus normal \citep{golub1999molecular,nguyen2002tumor,le2008sparse}. Indeed, if predictor variables are gene expression data measured by microarrays or by the RNA-seq technologies and the response is a phenotypic variables, situations affected by the BLLiM not only provides clusters of individuals based on the relation between gene expression data and the
phenotype but also implies a gene regulatory network specific for each cluster of individuals (see \cite{devijver2017nonlinear} for more details). 

It is worth noting that two hyperparameters must be estimated to construct a BLLiM model: the number of mixtures components (or clusters) and the block structure of large covariance matrices specific
of each cluster (the size and the number of blocks). Data driven choices of hyperparameters of learning algorithms belong to the model selection class of problems, which has attracted much attention in statistics and machine learning over the last 50 years \citep{akaike1974new,mallows2000some,anderson2004model,massart2007concentration}. This is a particular instance of the estimator (or model) selection problem: given a family of estimators, how do we choose, using data, one among them whose risk is as small as possible? Note that penalization is one of the main strategies proposed for model selection. It suggests to choose the estimator minimizing the sum of its empirical risk and some penalty terms corresponding to how well the model fits the data, while avoiding overfitting.

In general, model selection can be performed using the Akaike information criterion (AIC) or the Bayesian information criterion (BIC) \citep{akaike1974new,schwarz1978estimating}. Nevertheless, these approaches are asymptotic, which implies that there are no finite sample guarantees for choosing between different levels of complexity. Their use in small sample settings is thus ad hoc. To overcome such difficulties, \cite{birge2007minimal} proposed a novel approach, called slope heuristics, supported by a non-asymptotic oracle inequality. This method leads to an optimal data-driven
choice of multiplicative constants for penalties. Practical issues and recent surveys for the slope heuristic can be found in \cite{baudry2012slope}, and \cite{arlot2019minimal}.

It should be stressed that a general model selection result, originally established by \citet[Theorem 7.11]{massart2007concentration}, guarantees a penalized criterion leads to a good model selection and the penalty being defined by the model complexity. In particular, it provides support for the slope heuristic approach in a finite sample setting. Then, in the spirit of the concentration inequality-based methods developed in \cite{massart2007concentration}, \cite{Massart:2011aa}, and \cite{cohen2011conditional}, a huge number of finite-sample oracle results have been proposed in several statistical frameworks including high dimensional Gaussian graphical models \citep{devijver2018block}, Gaussian mixture model selection \citep{maugis2011non,maugis2011data}, finite mixture regression models \citep{Meynet:2013aa, Devijver:2015aa, devijver2015finite,devijver2017model,devijver2017joint}, soft-max-gated mixture of experts (SGaME) \citep{montuelle2014mixture,nguyen2020l1oracle}, and Gaussian-gated localized MoE (GLoME) models \citep{nguyen2021nonGLoME}. However, to the best of our knowledge, we are the first to
provide a finite-sample oracle inequality: \cref{weakOracleInequality}, for the BLoMPE regression model. In particular, our proof strategy makes use of recent novel approaches comprising a model selection theorem for maximum likelihood estimator (MLE) among a random subcollection \citep{devijver2015finite}, a non-asymptotic model selection result for detecting a good block-diagonal structure in high-dimensional graphical models \citep{devijver2018block} and a reparameterization trick to bound the metric entropy of the Gaussian gating parameter space in GLoME models \citep{nguyen2021nonGLoME}.

The main contribution of our paper is an important theoretical result: a finite-sample oracle inequality that provides a non-asymptotic bound on the risk, and a lower bound on the penalty function that ensures such non asymptotic theoretical control on the estimator under the Kullback--Leibler loss. It also provides a strong theoretical justification for the penalty shape when using the slope heuristic for the BLoMPE as well as BLLiM models.

The rest of this paper is organized as follows. In \cref{notationAndFramework}, we discuss the model construction and framework for BLoMPE and BLLiM models. Then, we present the main results of the paper, an oracle inequality satisfied by the penalized maximum likelihood of BLoMPE, in \cref{mainResult}.
\cref{sec.proofOracleIneq} is devoted to the proof of these main results based on a general model selection theorem. Proofs of lemmas are provided in \cref{proofLemma}.


\section{Notation and framework}\label{notationAndFramework}
We consider a regression framework and aim at capturing the potential  nonlinear relationship between the multivariate response $\Yv =\left(\Yv_j\right)_{j \in [L]},  [L]= \left\{1,\ldots,L\right\}$, and the set of covariates $\Xv=\left(\Xv_j\right)_{j \in [D]}$ with a potential hidden graph-structured interaction between covariates. 
Let $\left(\left(\Xv_i,\Yv_i\right)\right)_{ i \in [n]}\in \left(\cX\times\cY\right)^n \subset \left(\R^D\times\R^L\right)^n$ be a random sample, and let $\xv$ and $\yv$ denote the observed values of the random variables $\Xv$ and $\Yv$, respectively.

\subsection{BLoMPE models} \label{sec_BLoMPE_Model}

In order to accommodate a potential hidden graph-structured interaction and make a trade-off between complexity and sparsity, we consider an extension of the GLoME model from \cite{nguyen2021nonGLoME}, which generalized the MoE models \citep{jacobs1991adaptive,xu1995alternative}. More specifically, we consider the following BLoMPE model, defined by \eqref{eq_define_BLoMPE}, which is motivated by an inverse regression framework, where the role of response variables and high-dimensional predictors are exchanged such that the response $\Yv$ becomes the covariate and the predictor $\Xv$ plays the role of a multivariate response.

Then the BLoMPE model is defined by the following conditional density:
\begin{align}
	s_{\psib_{K,d}}(\xv| \yv) &= \sum_{k=1}^K \bfg_{k}\left(\yv;\omegab\right) \Phi_D\left(\xv;\upsilonb_{k,d}(\yv),\Sigmab_k\left(\bfB_k\right)\right),\label{eq_define_BLoMPE}
\end{align}
with	%
\begin{align}
	\bfg_{k}\left(\yv;\omegab\right)&= \frac{ \pib_k\Phi_L\left(\yv;\bfc_k,\Gammab_k\right)  }{\sum_{j=1}^K \pib_j \Phi_L\left(\yv;\bfc_j,\Gammab_j\right)}. \label{eq_BLoMPE_GaussianGating}
\end{align}
Here, $\bfg_k(\cdot; \omegab)$ and $\Phi_D\left(\cdot;\upsilonb_{k,d}(\cdot),\Sigmab_k\left(\bfB_k\right)\right), k\in[K]$, $K \in \Ns$, $d \in \Ns$, are called Gaussian gating functions and Gaussian experts, respectively. Furthermore, we decompose the parameters of the model as follows: $\psib_{K,d} = \left(\omegab,\upsilonb_d,\Sigmab\left(\bfB\right)\right) \in \Omegab_K\times \UpsilondK \times \bfV_K\left(\bfB\right) =:\Psib_{K,d} $, $\omegab = \left(\pib,\cb,\Gammab\right) \in \left(\Pib_{K-1} \times \bfC_K \times \Vb'_K\right) =: \Omegab_K$, $\pib = \left(\pib_k\right)_{k\in[K]}$, $\cb = \left(\bfc_k\right)_{k\in[K]}$, $\Gammab = \left(\Gammab_k\right)_{k\in[K]}$, $\upsilonb_d= \left(\upsilonb_{k,d}\right)_{k\in[K]} \in\UpsilondK $, and $\Sigmab\left(\bfB\right)= \left(\Sigmab_k\left(\bfB_k\right)\right)_{k\in[K]} \in \bfV_K\left(\bfB\right)$. Note that $\Pib_{K-1} =\left\{ \left(\pib_k\right)_{k\in[K]} \in \left(\R^+\right)^K, \sum_{k=1}^K \pib_k = 1\right\}$ is a $K-1$ dimensional probability simplex, $\bfC_K$ is a set of $K$-tuples of mean vectors of size $L \times 1$, $\Vb'_K$ is a sets of $K$-tuples of elements in $\cS_L^{++}$, where $\cS_L^{++}$ denotes the collection of symmetric positive definite matrices on $\R^L$, $\UpsilondK$ is a set of $K$-tuples of mean functions from $\R^L$ to $\R^D$ depending on a degree $d$ (\eg~a degree of polynomials), and $\bfV_K\left(\bfB\right)$ is a set containing $K$-tuples from $\cS_D^{++}$ with the following block-diagonal structures defined in \eqref{eq_block_diagonal_structures} \citep{devijver2017nonlinear,devijver2018block}.

More precisely, for $k \in [K]$, we decompose $\Sigmab_{k}\left(\bfB_k\right)$ into $G_k$ blocks, $G_k \in \Ns$, and we denote by $d^{[g]}_k$ the set of variables into the $g$th group, for $g\in \left[G_k\right]$, and by $\card\left(d^{[g]}_k\right)$ the number of variables in the corresponding set. Then, we denote by $\bfB_k = \left(d^{[g]}_k\right)_{g\in\left[G_k\right]}$ a block structure for the cluster $k$, and $\bfB = \left(\bfB_k\right)_{k\in [K]}$ the covariate indexes into each group for each cluster. Hence, up to a permutation, we can construct the following block-diagonal covariance matrices: $\bfV_K\left(\bfB\right) = \left(\bfV_k\left(\bfB_k\right)\right)_{k\in[K]}$, for every $k\in[K]$,
\begin{align} \label{eq_block_diagonal_structures}
	\bfV_k\left(\bfB_k\right) = \left\{ \Sigmab_{k}\left(\bfB_k\right) \in \mathbb{S}_D^{++}\left| 
	\begin{array}{l} 
		\Sigmab_{k}\left(\bfB_k\right) = \bfP_k
		\begin{pmatrix}
			\Sigmab_{k}^{[1]}&\zero&\ldots &\zero\\
			\zero&\Sigmab_{k}^{[2]}&\ldots&\zero\\
			\zero&\zero&\ddots&\zero\\
			\zero&\zero&\ldots&\Sigmab_{k}^{[G_k]}
		\end{pmatrix}
		\bfP^{-1}_k,\\
		\Sigmab_{k}^{[g]}\in \cS^{++}_{
			\card\left(d^{[g]}_k\right)}, \forall g \in [G_k]
	\end{array}
	\right.\right\},
\end{align}
where $\bfP_k$ corresponds to the permutation leading to a block-diagonal matrix in cluster $k$. It is worth mentioning that outside the blocks, all coefficients of the matrix are zeros and we also authorize reordering of the blocks: \eg$\left\{\left(1,3\right);\left(2,4\right)\right\}$ is identical to $\left\{\left(2,4\right);\left(1,3\right)\right\}$, and the permutation inside blocks: \eg the partition of $4$ variables into blocks $\left\{\left(1,3\right);\left(2,4\right)\right\}$ is the same as the partition $\left\{\left(3,1\right);\left(4,2\right)\right\}$. 

\begin{remark}
	The block-diagonal structures for covariance matrices $\left(\Sigmab_{k}\left(\bfB_k\right)\right)_{k\in[K]}$, defined in \eqref{eq_block_diagonal_structures}, are not only used for a trade-off between complexity and sparsity but also motivated by some real applications, where we want to perform prediction on data sets with heterogeneous observations and hidden graph-structured interactions between covariates. For instance, for gene expression data set in which conditionally on the phenotypic response, genes interact with few other genes only, \ie there are small modules of correlated genes (see	 \citealp{devijver2017nonlinear,devijver2018block} for more details).
\end{remark}

In order to establish our oracle inequality, \cref{weakOracleInequality}, we need to assume that $\cY$ is a bounded set in $\R^L$ and make explicit some classical boundedness conditions on the parameter space. 

\subsubsection{Gaussian gating functions}
For a matrix $\bfA$, let $m(\bfA)$ and $M(\bfA)$ be, respectively, the moduli of the smallest and largest eigenvalues of $\bfA$. We shall restrict our study to bounded Gaussian gating parameter vectors $\omegab = \left(\pib,\cb,\Gammab\right) \in \Omegab_K$. Specifically, we assume that there exist deterministic positive constants $a_\pib,A_{\cb},a_{\Gammab},A_{\Gammab}$, such that $\omegab$ belongs to $\widetilde{\Omegab}_K$, where
\begin{align} \label{define_BoundedGatesParameters}
	\widetilde{\Omegab}_K = \left\{\omegab \in \Omegab_K : \forall k \in [K], \left\|\bfc_k\right\|_\infty  \le A_{\cb},a_{\Gammab} \le m\left({\Gammab}_{k}\right) \le  M\left({\Gammab}_{k}\right) \le A_{\Gammab}, a_\pib \le \pib_k \right\}.
\end{align}
We denote the space of gating functions as
\begin{align*}
	\cP_K &= \left\{\bfg =\left(\bfg_{k}\left(\cdot;\omegab\right)\right)_{k \in [K]}: \forall k \in [K],\bfg_{k}\left(\yv;\omegab\right)= \frac{ \pib_k\Phi_L\left(\yv;\bfc_k,\Gammab_k\right)  }{\sum_{j=1}^K \pib_j \Phi_L\left(\yv;\bfc_j,\Gammab_j\right)}, \omegab \in \widetilde{\Omegab}_K\right\}.
\end{align*}

\subsubsection{Gaussian experts}
Following the same structure for the means of Gaussian experts from \cite{montuelle2014mixture,nguyen2021nonGLoME}, the set $\UpsilondK$ will be chosen as a tensor product of compact sets of moderate dimension (\eg~a set of polynomials of degree smaller than $d$, whose coefficients are smaller in absolute values than $T_\Upsilonb$). Then, $\UpsilondK$ is defined as a linear combination of a finite set of bounded functions whose coefficients belong to a compact set. This general setting includes polynomial bases when the covariates are bounded, Fourier bases on an interval, as well as suitably renormalized wavelet dictionaries. More specifically, $ \UpsilondK = \otimes_{k\in[K]}\Upsilonb_{k,d} =: \Upsilonb_{k,d}^K$, where $\Upsilonb_{k,d} = \Upsilonb_{b,d}$, $\forall k \in [K]$, and 
\begin{align} \label{eq_define_meanExperts_linearBounded}
	\Upsilonb_{b,d} &= \left\{ \yv \mapsto \left(\sum_{i=1}^{d} \alphab_i^{(j)} \varphi_{\Upsilonb,i}(\yv)\right)_{j\in [D]}=: \left(\upsilonb_{d,j}(\yv)\right)_{j\in [D]}:\norm{\alphab}_\infty \le T_\Upsilonb\right\}.
\end{align}
Here $d \in \Ns, T_\Upsilonb \in \R^+$, and $\left(\varphi_{\Upsilonb,i}\right)_{i \in \left[d\right]}$ is a collection of bounded functions on $\cY$. In particular, we focus on the bounded $\cY$ case and assume that $\cY = [0,1]^L$, without loss of generality. In this case, $\varphi_{\Upsilonb,i}$ can be chosen as monomials with maximum (non-negative) degree $d$: $\yv^{\bfr} = \prod_{l=1}^L \yv^{\bfr_l}_l$. Recall that a multi-index $\bfr = \left(\bfr_l\right)_{l\in[L]}, \bfr_l \in \Ns\cup\left\{0\right\}, \forall l \in [L]$, is an $L$-tuple of nonnegative integers. We define $\left|\bfr\right| = \sum_{l=1}^L \bfr_l$ and the number $|\bfr|$ is called the order or degree of $\yv^{\bfr}$. Then, $ \UpsilondK = \Upsilonb_{p,d}^K$, where
\begin{align}\label{eq_define_meanExperts_polynomial}
	\Upsilonb_{p,d} &= \left\{ \yv \mapsto \left(\sum_{|\bfr|=0}^{d} \alphab_\bfr^{(j)} \yv^{\bfr}\right)_{j\in [D]} =: \left(\upsilonb_{d,j}(\yv)\right)_{j\in [D]}:\norm{\alphab}_\infty \le T_\Upsilonb\right\}.
\end{align}

For the block-diagonal covariances of Gaussian experts, we assume that there exist some positive constants $\lambda_m$ and $\lambda_M$ such that, for every $k \in [K]$,
\begin{align}
	0<\lambda_m \le m\left(\Sigmab_{k}\left(\bfB_k\right)\right) \le M\left(\Sigmab_{k}\left(\bfB_k\right)\right) \le \lambda_M.
\end{align}
Note that this is a quite general assumption and is also used in the block-diagonal covariance selection for Gaussian graphical models of \cite{devijver2018block}.

Next, a characterization of BLLiM model, an affine instance of BLoMPE model, is described in \cref{randomCovModelSelection} and is especially useful for high-dimensional regression data. Note that the BLLiM model relies on an inverse regression trick from
a GLLiM model \citep{deleforge2015high} and the block-diagonal structure hypothesis on the residual covariance
matrices \citep{devijver2018block}.

\subsection{High-dimensional regression via BLLiM models}\label{randomCovModelSelection}
A BLLiM model, as originally introduced in \cite{devijver2017nonlinear}, 
is used to capture the nonlinear relationship between the response and the set of covariates, imposed by a potential hidden graph-structured interaction, from a high-dimensional regression data, typically in the case when $D\gg L$, by the following $K$ locally affine mappings:
\begin{align}\label{eq_locallyaffine}
	\Yv = \sum_{k=1}^K \Indi \left(Z = k\right) \left(\bfA^*_k\Xv +\bfb^*_k + \bfE^*_k\right).
\end{align}
Here, $\Indi$ is an indicator function and $Z$ is a latent variable capturing a cluster relationship, such that $Z=k$, if $\Yv$ originates from cluster $k \in [K]$. Cluster specific affine transformations are defined by matrices $\bfA^*_k \in \R^{L\times D}$ and vectors $\bfb^*_k \in \R^L$. Furthermore, $\bfE^*_k$ are an error terms capturing both the reconstruction error due to the local affine approximations and the observation noise in $\R^L$.

Following the common assumption that $\bfE^*_k$ is a zero-mean Gaussian variable with covariance matrix $\Sigmab_k^* \in \R^{L\times L}$, it holds that
\begin{align} \label{eq_conditional_forward}
	p\left(\Yv = \yv|\Xv = \xv,Z = k;\psib^*_K\right) = \Phi_L\left(\yv;\bfA^*_k\xv +\bfb^*_k,\Sigmab^*_k\right),
\end{align}
where we denote by $\psib^*_K$ the vector of model parameters and $\Phi_L$ is the probability density function (PDF) of a Gaussian distribution of dimension $L$. 
In order to enforce the affine transformations to be local, $\Xv$ is defined as a mixture of $K$ Gaussian components as follows:
\begin{align}\label{eq_marginal_forward}
	p\left(\Xv=\xv|Z = k;\psib^*_K\right) = \Phi_D\left(\xv;\bfc^*_k,\Gammab^*_k\right), p\left(Z=k;\psib^*_k\right) = \pib^*_k,
\end{align}
where $\bfc^*_k \in \R^D, \Gammab^*_k \in \R^{D \times D}$, $\pib^*=\left(\pib^*_k\right)_{k\in[K]} \in \Pib^*_{K-1}$, and $\Pib^*_{K-1}$
is the $K-1$ dimensional probability simplex.
Then, according to formulas for conditional multivariate Gaussian variables and the following hierarchical decomposition 
\begin{align*}
	p\left(\Yv = \yv,\Xv = \xv;\psib^*_K\right)
	&= \sum_{k=1}^K p\left(\Yv = \yv|\Xv = \xv,Z = k;\psib^*_K\right)p\left(\Xv=x|Z = k;\psib^*_K\right)p\left(Z=k;\psib^*_K\right),\nn\\
	&=  \sum_{k=1}^K \pib^*_k \Phi_D\left(\xv;\bfc^*_k,\Gammab^*_k\right) \Phi_L\left(\yv;\bfA^*_k\xv +\bfb^*_k,\Sigmab^*_k\right),
\end{align*}
we obtain the following forward conditional density \citep{deleforge2015high}:
\begin{align} \label{eq_forwardBLLiM}
	p\left(\Yv = \yv|\Xv = \xv;\psib^*_K\right) = \sum_{k=1}^K\frac{ \pib^*_k \Phi_D\left(\xv;\bfc^*_k,\Gammab^*_k\right)  }{\sum_{j=1}^K \pib^*_j \Phi_D\left(\xv;\bfc^*_j,\Gammab^*_j\right)}\Phi_L\left(\yv;\bfA^*_k\xv +\bfb^*_k,\Sigmab^*_k\right),
\end{align}
where $\psib^*_K =  \left(\pib^*,\thetab^*_K \right) \in \Pi_{K-1} \times \Thetab^*_K =:\Psib^*_K$. Here, $\thetab^*_K=\left( \bfc^*_k,\Gammab^*_k,\bfA^*_k,\bfb^*_k,\Sigmab^*_k\right)_{k\in[K]}$ and 
$$\Thetab^*_K = \left(\R^D\times \cS_D^{++} (\R) \times \R^{L\times D}\times \R^L\times \cS_L^{++} (\R)\right)^K.$$
Without assuming anything on the structure on of parameters, the dimension of the model (denoted by $\dim\left(\cdot\right)$), is defined as the total number of parameters that has to be estimated, as follows:
\begin{align*}
	\dim\left(\Psib^*_K\right) = K\left(1+D(L+1)+\frac{D(D+1)}{2}+\frac{L(L+1)}{2}+L\right)-1.
\end{align*}

It is worth mentioning that $	\dim\left(\Psib_K\right)$ is large compared to the sample size (see, e.g., \citealp{deleforge2015high,devijver2017nonlinear,perthame2018inverse} for more details regarding their real data sets), whenever $D$ is large and $D\gg L$. Furthermore, it is more realistic to make assumption on the residual covariance matrices $\Sigmab^*_k$ of error vectors $\bfE^*_k$ rather than on $\Gammab^*_k$ (cf. \citealp[Section 3]{deleforge2015high}). This justifies the use of the inverse regression trick from \cite{deleforge2015high}, which leads a drastic reduction in the number of parameters to be estimated.

More specifically, in \eqref{eq_forwardBLLiM}, the roles of input and response variables should be exchanged such that $\Yv$ becomes the covariates and $\Xv$ plays the role of the multivariate response. Therefore, its corresponding inverse conditional density is defined as a block-diagonal covariance for Gaussian locally-linear mapping (BLLiM) model, based on the previous hierarchical Gaussian mixture model, as follows:
\begin{align}
	p\left(\Xv = \xv|\Yv = \yv,Z = k;\psib_K\right) &= \Phi_D\left(\xv;\Ab_k\yv +\bb_k,\Sigmab_k\right), \label{eq_conditional_inverse}
	\\
	p\left(\Yv=\yv|Z = k;\psib_K\right) &= \Phi_L\left(\yv;\bfc_k,\Gammab_k\right), p\left(Z=k;\psib_k\right) = \pib_k,\label{eq_marginal_inverse}\\
	p\left(\Xv = \xv|\Yv = \yv;\psib_K\right) &= \sum_{k=1}^K\frac{ \pib_k \Phi_L\left(\yv;\bfc_k,\Gammab_k\right)  }{\sum_{j=1}^K \pib_j \Phi_L\left(\yv;\bfc_j,\Gammab_j\right)}\Phi_D\left(\xv;\Ab_k\yv +\bb_k,\Sigmab_k\right)\label{eq_inverseBLLiM},
\end{align}
where $\Sigmab_k$ is a $D \times D$ block-diagonal covariance structure automatically learnt from data and $\psib_K$ is the set
of parameters $\psib_K =  \left(\pib_K,\thetab_K \right) \in \Pib_{K-1} \times \Thetab_K =: \Psib_K$. It is important to note that the BLLiM model imposes the block-diagonal structures on $\left(\Sigmab_k\right)_{k\in[K]}$ to make a trade-off between complexity and sparsity.

An intriguing feature of the GLLiM model is described in \cref{lem_inverse_regression_mapping}.

\begin{lem}[Lemma 2.1 from \cite{nguyen2021nonGLoME}]\label{lem_inverse_regression_mapping}
	The parameter $\psib^*_K$ in the forward conditional PDF, defined in \eqref{eq_forwardBLLiM}, can then be deduced from $\psib_K$ in \eqref{eq_inverseBLLiM} via the following one-to-one correspondence:
	\begin{align*} 
		\thetab_K=\begin{pmatrix}
			\bfc_k\\
			\Gammab_k \\
			\bfA_k \\
			\bfb_k \\
			\Sigmab_k
		\end{pmatrix}_{k \in [K]}
		\mapsto
		\begin{pmatrix}
			\bfc_k^*\\
			\Gammab_k^* \\
			\bfA^*_k \\
			\bfb^*_k \\
			\Sigmab_k^*
		\end{pmatrix}_{k \in [K]}
		=
		\begin{pmatrix}
			\bfA_k\bfc_k +\bfb_k  \\
			\Sigmab_k+\bfA_k\Gammab_k\bfA_k^\top\\
			\Sigmab_k^*\bfA_k^\top\Sigmab_k^{-1}\\
			\Sigmab_k^* (\Gammab_k^{-1}\bfc_k-\bfA_k^\top\Sigmab_k^{-1}\bfb_k)\\
			\left(\Gammab_k^{-1}+\bfA_k^\top\Sigmab_k^{-1}\bfA_k\right)^{-1}
		\end{pmatrix}_{k \in [K]}
		\in \Thetab^*_K,
	\end{align*}
	with the note that $\pib^* \equiv \pib$.
\end{lem}

\subsection{Collection of BLoMPE models} \label{sec.collectionOfModels}
In this paper, we choose the degree of polynomials $d$ and the number of components $K$ among finite sets $\cD_\Upsilonb = \left[d_{\max}\right]$ and $\cK = \left[K_{\max}\right]$, respectively, where $d_{\max}\in \Ns$ and $K_{\max}\in \Ns$ may depend on the sample size $n$. Moreover, $\bfB$ is selected among a list of candidate structures $\left(\cB_k\right)_{k \in [K]} \equiv \left(\cB\right)_{k \in [K]}$, where $\cB$ denotes the set of all possible partitions of the covariables indexed by $[D]$, for each cluster of individuals. We wish to estimate the unknown conditional density $s_0$ by conditional densities belonging to the following collection of models:  $ \left(S_\bfm\right)_{\bfm \in \cM}$, $\cM = \left\{\left(K,d,\bfB\right): K \in \cK,d \in \cD_\Upsilonb, \bfB \in \left(\cB\right)_{k \in [K]}\right\}$,
\begin{align} \label{defineModelBoundedBlock}
	S_\bfm 
	=  \left\{(\xv,\yv) \mapsto
	s_{\psib_{K,d}}(\xv| \yv): \psib_{K,d} = \left(\omegab,\upsilonb_{d},\Sigmab\left(\bfB\right)\right) \in \widetilde{\Omegab}_K\times \UpsilondK \times \bfV_K\left(\bfB\right) =:\widetilde{\Psib}_{K,d}\left(\bfB\right)\right\}.
\end{align}
In theory, we would like to consider the whole collection of model $ \left(S_\bfm\right)_{\bfm \in \cM}$.
However, the cardinality of $\cB$ is large; its size is a Bell number. Even for a moderate number of variables $D$, it is not possible to explore the set $\cB$, exhaustively. We restrict our attention to a random subcollection $\cB^R$ of moderate size. For example, we can consider the BLLiM procedure from \citet[Section 2.2]{devijver2017nonlinear}.

Next, in \cref{sec_PMLE_KullbackLeibler}, we introduce a general principle of PMLE and the losses, Kullback--Leibler and Jensen--Kullback--Leibler divergences, that are useful to compare two conditional probability density functions.
\subsection{Penalized maximum likelihood estimator and losses}\label{sec_PMLE_KullbackLeibler}
In the context of PMLE, given the collection of conditional densities $S_\bfm$, we aim to estimate $s_0$ by the $\eta$-minimizer $\widehat{s}_\bfm$ of the negative log-likelihood (NLL):
\begin{align}\label{eq_define_NLL}
	\widehat{s}_\bfm = \argmin_{s_\bfm \in S_\bfm} \sum_{i=1}^n -\ln \left(s_\bfm \left(\xv_i | \yv_i \right)\right) + \eta,
\end{align}
where the error term $\eta$ is necessary when the infimum may not be unique or even not be reached.

As always, using the NLL of the estimate in each model as a criterion is not sufficient. It is an underestimation of the risk of the estimate and this leads to choosing models that are too complex. By adding a suitable penalty $\pen(\bfm)$, one hopes to co	mpensate between the variance term,  $\frac{1}{n}\sum_{i=1}^n -\ln \frac{\widehat{s}_\bfm \left(\xv_i | \yv_i \right)}{s_0 \left(\xv_i | \yv_i \right)} $, and the bias, $\inf_{s_\bfm \in S_\bfm} \tkl\left(s_0,s_\bfm\right)$. For a given choice of $\pen(\bfm)$, the best model $S_{\widehat{\bfm}}$
is chosen as the one whose index is an $\eta'$-almost minimizer of the sum of the NLL and this penalty:
\begin{align}\label{eq_define_penalizedNLL}
	\widehat{\bfm} = \argmin_{\bfm \in \cM} \left(\sum_{i=1}^n -\ln \left(\widehat{s}_\bfm \left(\xv_i | \yv_i\right) \right)+ \pen(\bfm)\right) + \eta'.
\end{align}
Note that $\widehat{s}_{\widehat{\bfm}}$ is then called the $\eta'$-penalized likelihood estimate and depends on both the error terms $\eta$ and $\eta'$. From hereon in, the term best model or estimate is used to indicate that it satisfies  \eqref{eq_define_penalizedNLL}.

In the maximum likelihood approach, the Kullback--Leibler divergence is the most natural loss function, which is defined for two densities $s$ and $t$ by
\begin{align*}
	\kl(s,t) = \begin{cases}
		\int_{\R^D} \ln\left(\frac{s(y)}{t(y)}\right)s(y)dy& \text{ if $sdy$ is absolutely continuous \wrt~$tdy$},\\
		+\infty & \text{ otherwise}.
	\end{cases}
\end{align*}
However, to take into account the structure of conditional
densities and the random covariates $\left(\Yv_i\right)_{i \in [n]}$, we consider the tensorized Kullback--Leibler divergence $\tkl$, defined as:
\begin{align} \label{tensorizedKLD}
	\tkl(s,t) = \E{\Yv}{\frac{1}{n} \sum_{i=1}^n \kl\left(s\left(\cdot| \Yv_i\right),t\left(\cdot| \Yv_i\right)\right)},
\end{align}
if $sdy$ is absolutely continuous \wrt~$tdy$, and $+\infty$ otherwise. Note that if the predictors are fixed, this divergence
is the classical fixed design type divergence in which there is no expectation. We refer to our result as a weak oracle inequality, because its statement is based on a smaller divergence, when compared to $\tkl$, namely the tensorized Jensen--Kullback--Leibler divergence:
\begin{align*}
	\jtkl(s,t) = \E{\Yv}{\frac{1}{n} \sum_{i=1}^n \frac{1}{\rho}\kl\left(s\left(\cdot| \Yv_i\right),\left(1-\rho\right)s\left(\cdot| \Yv_i\right)+\rho t\left(\cdot| \Yv_i\right)\right)},
\end{align*}
with $\rho \in \left(0,1\right)$. We note that $\jtkl$ was first used in \cite{cohen2011conditional}. However, a version of this divergence appears explicitly with $\rho = \frac{1}{2}$ in \cite{massart2007concentration}, and it is also found implicitly in \cite{birge1998minimum}. This loss is always bounded by $\frac{1}{\rho}\ln \frac{1}{1-\rho}$ but behaves like $\tkl$, when $t $ is close to $s$. The main tools in the proof of such a weak oracle inequality are deviation inequalities for sums of random variables and their suprema. These tools require a boundedness assumption on the controlled functions, which is not satisfied by $-\ln\frac{s_\bfm}{s_0}$, and thus also not satisfied by $\tkl$. Therefore, we consider instead the use of $\jtkl$. In particular, in general, it holds that  $C_\rho \thel \le \jtkl \le \tkl$, where  $C_\rho = \frac{1}{\rho}\min\left(\frac{1-\rho}{\rho},1\right)\left(\ln\left(1+\frac{\rho}{1-\rho}\right)-\rho\right)$ (see \citealt[Prop. 1]{cohen2011conditional}) and $\thel$ is a tensorized extension of the squared Hellinger distance $\thel$, defined by
\begin{align*}
	\thel(s,t) = \E{\Yv}{\frac{1}{n} \sum_{i=1}^n \hel\left(s\left(\cdot| \Yv_i\right), t\left(\cdot| \Yv_i\right)\right)}.
\end{align*}
Moreover, if we assume that, for any $\bfm \in \cM$ and any $s_\bfm \in S_\bfm, s_0 d\lambda \ll s_\bfm d\lambda$, then (cf. \citealp{montuelle2014mixture,cohen2011conditional})
\begin{align} \label{montuelle.KLandJKL.1}
	\frac{C_\rho}{2+\ln \norm{s_0/s_\bfm}}_\infty \tkl(s_0,s_\bfm) \le \jtkl(s_0, s_\bfm).
\end{align} 

In \cref{mainResult}, we state our main contribution: a finite-sample oracle type inequality, which ensures that if we have penalized the log-likelihood in an approximate approach, we are able to select a model, which is as good as the oracle.

\section{Main result on oracle inequality} \label{mainResult}

Note that in this article, the constructed collection of models with block-diagonal structures for each cluster of individuals is designed, for example, by the BLLiM procedure from \cite{devijver2017nonlinear}, where each collection of partition is sorted by sparsity level. Nevertheless, our finite-sample oracle inequality still holds for any random subcollection of $\cM$, which is constructed by some suitable tools in the framework of BLoMPE regression models.

\begin{theorem}[Oracle inequality]\label{weakOracleInequality}
	Let $(\xv_i,\yv_i)_{i \in [n]}$ be the observations coming from an unknown conditional density $s_0$. For each $\bfm = \left(K,d,\bfB\right) \in \left(\cK \times \cD_\Upsilonb \times \cB\right)\equiv \cM$, let $S_\bfm$ be define by \eqref{defineModelBoundedBlock}. Assume that there exists $\tau > 0$ and $\epsilon_{KL} > 0$ such that, for all $\bfm \in \cM$, one can find $\bar{s}_\bfm \in S_\bfm$, such that
	\begin{align*}
		\tkl\left(s_0,\bar{s}_\bfm\right) \le \inf_{t \in  S_\bfm} \tkl\left(s_0,t\right) + \frac{\epsilon_{KL}}{n}, \text{ and } \bar{s}_\bfm\ge e^{-\tau} s_0.
	\end{align*}
	Next, we construct some random subcollection $\left(S_\bfm\right)_{\bfm \in \widetilde{\cM}}$ of $\left(S_\bfm\right)_{\bfm \in \cM}$ by letting $\widetilde{\cM}\equiv\left(\cK \times \cD_\Upsilonb \times \cB^R\right) \subset \cM$ such that $\cB^R$ is a random subcollection $\cB$, of moderate size, as described in \cref{sec.collectionOfModels}. Consider the collection $\left(\widehat{s}_\bfm\right)_{\bfm \in \widetilde{\cM}}$ of $\eta$-log likelihood minimizers satisfying \eqref{eq_define_NLL} for all $\bfm \in \widetilde{\cM}$.
	Then, there is a constant $C$ such that for any $\rho \in (0,1)$, and any $C_1 > 1$, there are two constants $\kappa_0$ and $C_2$ depending only on $\rho$ and $C_1$ such that, for every index, $\bfm \in \cM$, $ \xi_\bfm \in \R^+$, $ \Xi = \sum_{\bfm \in \cM}e^{-\xi_\bfm} < \infty$ and 
	\begin{align*}
		\pen(\bfm) \ge \kappa \left[\left(C+\ln n\right)\dim(S_\bfm) + (1 \vee \tau)\xi_\bfm\right],
	\end{align*}
	with $\kappa > \kappa_0$, the $\eta'$-penalized likelihood estimate $\widehat{s}_{\widehat{\bfm}}$, defined as in \eqref{eq_define_penalizedNLL} on the subset $\widetilde{\cM} \subset \cM$,
	satisfies
	\begin{align*}
		\E{}{\jtkl\left(s_0,\widehat{s}_{\widehat{\bfm}}\right)}&\le C_1\E{}{\inf_{\bfm \in \widetilde{\cM}}\left(\inf_{t \in S_\bfm} \tkl\left(s_0,t\right)+2 \frac{\pen(\bfm)}{n}\right)}
		+ C_2(1 \vee \tau)\frac{\Xi^2}{n} + \frac{\eta' + \eta}{n}.
	\end{align*}
\end{theorem}

\begin{remark}
	In \cref{weakOracleInequality}, the finite-sample oracle inequality compares performances of our estimator with the best model in the collection. However, \cref{weakOracleInequality} allows us to approximate well a rich class of conditional densities if we take enough degree of polynomials of Gaussian expert means and/or enough clusters in the context of mixture of Gaussian experts \citep{jiang1999hierarchical,mendes2012convergence,nguyen2016universal,ho2019convergence,nguyen2020approximationMoE}. This leads to the term on the right hand side being small, for $\cD_\Upsilonb$ and $\cK$ well-chosen.
	
	Furthermore, in the context of MoE regression models, our non-asymptotic oracle inequality, \cref{weakOracleInequality}, can be considered as a complementary result to a classical asymptotic theory \citep[Theorems 1,2, and 3]{khalili2010new}, to a finite-sample oracle inequality on the whole collection of models \citep{montuelle2014mixture,nguyen2021nonGLoME} and to an $l_1$-oracle inequality focusing on the Lasso estimation properties rather than the model selection procedure \citep{nguyen2020l1oracle}.
	
	In particular, our finite-sample oracle inequality resolves completely the following two important problems in the area of MoE regression models: (1) What number of mixture components $K$ should be chosen, given the size $n$ of the training data, and (2) Whether it is better to use a few complex experts or combine many simple experts, given
	the total number of parameters. Note that, such problems are considered in the work of \cite{mendes2012convergence}, where the authors provided some qualitative insights and suggested a practical method for choosing $K$ and $d$ involving a complexity penalty or cross-validation. Furthermore, their model is only for a non-regularized maximum-likelihood estimation and thus is not suitable in the high-dimensional setting.
\end{remark}


\section{Proof of the oracle inequality} \label{sec.proofOracleIneq}

\paragraph{Sketch of the proof}
To work with conditional density estimation in the BLoMPE regression models, in \cref{sec_general_model_selection_random}, we need to present a general theorem for model selection: \cref{thm_5_1_devijver2015finite}.
It is worth mentioning that, because the model collection constructed by the BLLiM procedure is random, we have to use a model selection theorem for MLE among a random subcollection (cf. \citealp[Theorem 5.1]{devijver2015finite} and \citealp[Theorem 7.3]{devijver2018block}), which is an extension of a whole collection of conditional densities from \citet[Theorem 2]{cohen2011conditional}, and of \citet[Theorem 7.11]{massart2007concentration}, working only for density estimation.
Then, we explain how we use \cref{thm_5_1_devijver2015finite} to get the oracle inequality: \cref{weakOracleInequality} in \cref{sec_proof_weakOracleInequality}. To this end, our model collection has to satisfy some regularity assumptions, which are proved in \cref{proofLemma}. The main difficulty in proving our oracle inequality lies in bounding the bracketing entropy of the Gaussian gating functions of the BLoMPE model and Gaussian experts with block-diagonal covariance matrices. To overcome the former issue, we follow a reparameterization trick of the Gaussian gating parameters space \citep{nguyen2021nonGLoME}. For the second one, we utilize the recent novel result on block-diagonal covariance matrices in \cite{devijver2018block}.


\subsection{Model selection theorem for MLE among a random subcollection} \label{sec_general_model_selection_random}

Before stating the general theorem, we begin by discussing our assumptions. We work here in a more general context, with $(\Xv,\Yv) \in \cX \times \cY$, and $\left(S_\bfm\right)_{\bfm \in \cM}$ defining a model collection indexed by $\cM$. 

First, we impose a structural assumption on each model indexed by $\bfm\in \cM$, regarding the bracketing entropy, defined by \eqref{eq_definition_BracketingEntropy}, conditioned on the model $S_\bfm$ \wrt~the tensorized Hellinger divergence. Recall that the bracketing entropy of a set $S$ with respect to any distance $d$, denoted by $\cH_{\left[\cdot\right],d}(\left(\delta,S\right))$, is defined as the logarithm of the minimal number $\cN_{\left[\cdot\right],d}\left(\delta,S\right)$ of brackets $\left[t^{-},t^{+}\right]$ covering $S$, such that $d(t^-,t^+) \le \delta$. That is,
\begin{align}\label{eq_definition_BracketingEntropy}
	\cN_{\left[\cdot\right],d}\left(\delta,S\right):= \min\left\{n \in \Ns : \exists t^-_1,t^+_1,\ldots,t^-_n,t^+_n \text{ s.t }d(t^-_k,t^+_k) \le \delta,S \subset \bigcup_{k=1}^n \left[t^-_k,t^+_k\right] \right\},
\end{align}
where the bracket $s \in \left[t^-_k,t^+_k\right]$ is defined by $t^{-}_k(\xv,\yv) \le s(\xv,\yv)\le t^{+}_k(\xv,\yv)$, $\forall (\xv,\yv) \in \cX \times \cY$.

This leads to the \cref{assumption_H} (H).
\begin{assumption}[H]\label{assumption_H}
	For every model $S_\bfm$ in the collection $\cS$, there is a non-decreasing function $\phi_\bfm$ such that $\delta \mapsto \frac{1}{\delta}\phi_\bfm(\delta)$ is non-increasing on $\left(0,\infty\right)$ and for every $\delta \in \R^+$,
	\begin{align*}
		\int_{0}^\delta \sqrt{\cH_{\left[.\right],\thell}\left(\delta,S_\bfm\left(\widetilde{s},\delta\right)\right)}d\delta \le  \phi_\bfm(\delta),
	\end{align*}
	where $S_\bfm\left(\widetilde{s},\delta\right) = \left\{s_\bfm \in S_\bfm : \thell \left(\widetilde{s},s_\bfm\right) \le  \delta\right\}$. The model complexity $\cD_\bfm$ of $S_\bfm$ is then defined as $n \delta^2_\bfm$, where $\delta_\bfm$ is the unique root of $\frac{1}{\delta} \phi_\bfm(\delta) = \sqrt{n} \delta$. 
\end{assumption} 
This bracketing entropy integral, often called the Dudley integral, plays an important role in empirical processes theory (cf. \citealp{van1996weak,kosorok2007introduction}). Observe that the model complexity does not depend on the bracketing entropies of the global models $S_\bfm$, but rather on those of smaller localized sets $S_\bfm\left(\widetilde{s},\delta\right)$.

For technical reasons, a seperability assumption, which always satisfied in the setting of this paper, is also required. \cref{assumption_seperability} (Sep) is a mild condition, which is classical in empirical process theory \citep{van1996weak} and allows us to work with a countable subset of $S_\bfm$. 
\begin{assumption}[Sep]\label{assumption_seperability}
	For every model $S_\bfm$, there exists some countable subset $S'_\bfm$ of $S_\bfm$ and a set $\cY'_\bfm$ with $\iota \left(\cY \setminus \cY'_\bfm\right) = 0$, where $\iota$ denotes Lebesgue measure, such that for every $t \in S_\bfm$, there exists some sequence $\left(t_k\right)_{k \geq 1}$ of elements of $S'_\bfm$, such that for every $x \in \cX$ and every $y \in \cY'_\bfm, \ln\left(t_k\left(y | x\right)\right) \xrightarrow{k \rightarrow + \infty}  \ln\left(t\left(y | x\right)\right)$.
\end{assumption} 
Furthermore, we also need an information theory type assumption to control the complexity of our collection. We assume the existence of a Kraft-type inequality for the collection \citep{massart2007concentration,barron2008mdl}.
\begin{assumption}[K] \label{assumption.Kraft-type}
	There is a family $\left(\xi_\bfm\right)_{\bfm \in \cM}$ of non-negative numbers and a real number $\Xi$, such that
	\begin{align*}
		\sum_{\bfm \in \cM}e^{-\xi_\bfm} \le \Xi < + \infty.
	\end{align*}
\end{assumption}

We can now state the main result of \citep[Theorem 5.1]{devijver2015finite} for the model selection theorem for MLE among a random subcollection.
\begin{theorem}[Theorem 5.1 from \cite{devijver2015finite}]\label{thm_5_1_devijver2015finite}
	Let $(\xv_i,\yv_i)_{i \in [n]}$ be observations coming from an unknown conditional density $s_0$. Let the model collection $\cS = \left(S_\bfm\right)_{\bfm \in \cM}$ be an at most countable collection of conditional density sets. Assume that \cref{assumption_H} (H), \cref{assumption_seperability} (Sep), and \cref{assumption.Kraft-type} (K) hold for every $\bfm \in \cM$. Let $\epsilon_{KL} > 0$, and $\bar{s}_\bfm \in S_\bfm$, such that
	\begin{align*}
		\tkl\left(s_0,\bar{s}_\bfm\right) \le \inf_{t \in  S_\bfm} \tkl\left(s_0,t\right) + \frac{\epsilon_{KL}}{n};
	\end{align*}
	and let $\tau > 0$, such that
	\begin{align}\label{eq.assumptionRS}
		\bar{s}_\bfm\ge e^{-\tau} s_0.
	\end{align}
	Introduce $\left(S_\bfm\right)_{\bfm \in \widetilde{\cM}}$, a random subcollection of $\left(S_\bfm\right)_{\bfm \in \cM}$.
	Consider the collection $\left(\widehat{s}_\bfm\right)_{\bfm \in \widetilde{\cM}}$ of $\eta$-log likelihood minimizer satisfying \eqref{eq_define_NLL} for all $\bfm \in \widetilde{\cM}$.
	Then, for any $\rho \in (0,1)$, and any $C_1 > 1$, there are two constants $\kappa_0$ and $C_2$ depending only on $\rho$ and $C_1$, such that, for every index $\bfm \in \cM$,
	\begin{align*}
		\pen(\bfm) \ge \kappa \left(\cD_\bfm + (1 \vee \tau)\xi_\bfm\right),
	\end{align*}
	with $\kappa > \kappa_0$, and where the model complexity $\cD_\bfm$ is defined in \cref{assumption_H},
	the $\eta'$-penalized likelihood estimate $\widehat{s}_{\widehat{\bfm}}$, defined as in \eqref{eq_define_penalizedNLL} on the subset $\widetilde{\cM} \subset \cM$,
	satisfies
	\begin{align*}
		\E{}{\jtkl\left(s_0,\widehat{s}_{\widehat{\bfm}}\right)}&\le C_1\E{}{\inf_{\bfm \in \widetilde{\cM}}\left(\inf_{t \in S_\bfm} \tkl\left(s_0,t\right)+2 \frac{\pen(\bfm)}{n}\right)} + C_2(1 \vee \tau)\frac{\Xi^2}{n} + \frac{\eta' + \eta}{n}.
	\end{align*}
\end{theorem}

In the next section, we apply \cref{thm_5_1_devijver2015finite} to prove \cref{weakOracleInequality}. Consequently, the penalty can be chosen roughly proportional to the intrinsic dimension of the model, and thus of the order of the variance.


\subsection{Proof of Theorem \ref{weakOracleInequality}}\label{sec_proof_weakOracleInequality}
It should be stressed that all we need is to verify that \cref{assumption_H} (H), \cref{assumption_seperability} (Sep) and \cref{assumption.Kraft-type} (K) hold for every $\bfm \in \cM$. According to the result from \citet[Section 5.3]{devijver2015finite}, \cref{assumption_seperability} (Sep) holds when we consider Gaussian densities and the assumption defined by \eqref{eq.assumptionRS} is true if we assume further that the true conditional density $s_0$ is bounded and compactly supported. Furthermore, since we restricted $d$ and $K$ to $\cD_\Upsilonb = \left[d_{\max}\right]$ and $\cK = \left[K_{\max}\right]$, respectively, it is true that there exists a family $\left(\xi_\bfm\right)_{\bfm \in \cM}$ and $\Xi >0$ such that, \cref{assumption.Kraft-type} (K) is satisfied.
Therefore, the main steps of the proof for the remaining \cref{assumption_H} (H) are presented in this \cref{sec_proof_weakOracleInequality}. All technical results are deferred to \cref{proofLemma}.

Note that the definition of complexity of model $S_\bfm$ in \cref{assumption_H} (H) is related to a classical entropy dimension of a compact set \wrt~a Hellinger type divergence $\thell$, thanks to the following \cref{prop.DudleyClassicalEntropy}, which is established in \citep[Proposition 2]{cohen2011conditional}.
\begin{prop}[Proposition 2 from \cite{cohen2011conditional}]\label{prop.DudleyClassicalEntropy}
	If, for any $\delta \in (0,\sqrt{2}]$, $ \cH_{\left[.\right],\thell}\left(\delta,S_\bfm\right) \le \dim(S_\bfm)\left(C_\bfm+\ln\left(\frac{1}{\delta}\right)\right)$, then the function 
	\begin{align*}
		\phi_\bfm\left(\delta\right) = \delta \sqrt{\dim\left(S_\bfm\right)}\left(\sqrt{C_\bfm}+\sqrt{\pi}+\sqrt{\ln\left(\frac{1}{\min\left(\delta,1\right)}\right)}\right)
	\end{align*} satisfies \cref{assumption_H} (H). Furthermore, the unique solution $\delta_\bfm$ of $\frac{1}{\delta} \phi_\bfm\left(\delta\right) =  \sqrt{n} \delta$ satisfies
	\begin{align*}
		n \delta^2_\bfm \le \dim(S_\bfm) \left(2 \left(\sqrt{C_\bfm} + \sqrt{\pi}\right)^2 + \left(\ln \frac{n}{\left(\sqrt{C_\bfm} + \sqrt{\pi}\right)^2\dim\left(S_\bfm\right)}\right)_+\right).
	\end{align*}
\end{prop}
Then, \cref{assumption_H} (H) is proved via \cref{prop.DudleyClassicalEntropy} using the fact that
\begin{align}\label{eq.bractketingEntropyModelSm}
	\cH_{\left[.\right],\thell}\left(\delta,S_\bfm\right) \le \dim(S_\bfm)\left(C_\bfm+\ln\left(\frac{1}{\delta}\right)\right),
\end{align}
where $C_\bfm$ is a constant depending on the model. Before proving the previous statement \eqref{eq.bractketingEntropyModelSm}, we need to define the following distance over conditional densities:
\begin{align*}
	\sup_\yv d_\xv(s,t) = \sup_{\yv \in \cY} \left(\int_\cX \left(\sqrt{s(\xv| \yv)} - \sqrt{t(\xv| \yv)}\right)^2 d\xv \right)^{1/2}.
\end{align*}
This leads straightforwardly to $\thel(s,t) \le  \sup_\yv d^2_\xv(s,t)$. Then, we also define
\begin{align*}
	\sup_\yv d_k\left(\bfg,\bfg'\right) = \sup_{\yv\in \cY} \left(\sum_{k=1}^K \left(\sqrt{\bfg_k(\yv)}-\sqrt{\bfg'_k(\yv)}\right)^2\right)^{1/2},
\end{align*}
for any gating functions $\bfg$ and $\bfg'$.
To this end, given any densities $s$ and $t$ over $\cX$, the following distances, depending on $\yv$, is constructed as follows:
\begin{align*}
	\sup_\yv \max_k d_\xv(s,t) &= \sup_{\yv \in \cY} \max_{k \in [K]} d_\xv\left(s_k(\cdot,\yv),t_k(\cdot,\yv)\right)
	=\sup_{\yv \in \cY} \max_{k \in [K]} \left(\int_\cX \left(\sqrt{s_k(\xv,\yv)}-\sqrt{t_k(\xv,\yv)}\right)^2d\xv\right)^{1/2}.
\end{align*}

Then \eqref{eq.bractketingEntropyModelSm} can be established by first decomposing the entropy term between the Gaussian gating functions and the Gaussian experts. Indeed, there are two possible ways to decompose the bracketing entropy of $S_\bfm$ based on the reparameterization trick \citep{nguyen2021nonGLoME}, for $\cP_K$ via $\cW_k$ and Gaussian experts $\gkdb$.
\begin{align*}
	\cW_K &= \left\{\cY \ni \yv \mapsto \left(\ln\left(\pib_k\Phi\left(\yv;\bfc_k,\Gammab_k\right)\right)\right)_{k \in [K]} =: \left(\bfw_{k}(\yv;\omegab)\right)_{k \in [K]}= \bfw\left(\yv;\omegab\right) : \omegab \in \widetilde{\Omegab}_K\right\},\\
	\cP_K&= \left\{\cY \ni \yv \mapsto\left(\frac{e^{\bfw_k(\yv)}}{\sum_{l=1}^K e^{\bfw_{l}(\yv)}}\right)_{k\in[K]} =: \left(\bfg_{k}\left(\yv;\bfw\right)\right)_{k \in [K]}, \bfw \in \cW_K\right\} 	\text{, and }\\
	\gkdb &= \left\{\cX \times \cY \ni (\xv,\yv) \mapsto \left(\Phi\left(\xv;\upsilonb_{k,d}(\yv),\Sigmab_k\left(\bfB_k\right)\right)
	\right)_{k \in [K]}: \upsilonb_{d} \in \UpsilondK, \Sigmab(\bfB) \in \bfV_K(\bfB) \right\}.
\end{align*}
For the first approach, we can use  \cref{lem.bracketingEntropyDecomposition} \citep[Lemma 5]{montuelle2014mixture}:
\begin{lem}\label{lem.bracketingEntropyDecomposition}
	For all $\delta \in (0,\sqrt{2}]$ and $\bfm \in \cM$,
	\begin{align*}
		\cH_{[\cdot],\sup_\yv d_\xv }\left(\delta,S_\bfm\right) \le 	\cH_{[\cdot],\sup_\yv d_k }\left(\frac{\delta}{5},\cP_K\right) + 	\cH_{[\cdot],\sup_\yv\max_k d_\xv }\left(\frac{\delta}{5},\gkdb\right).
	\end{align*}
\end{lem}
Note that \cref{lem.bracketingEntropyDecomposition} boils down to assuming that $\Yv$ is bounded. To weaken this assumption, we are going to use the smaller distance: $\thell$, for the entropy with bracketing of $S_\bfm$ although bounding, such bracketing entropies for $\cW_K$ and $\cG_{K,\cB}$ becomes much more challenging. Consequently, this leads to the second approach via \cref{lem_bracketingEntropyDecomposition2} \citep[Lemma 6]{montuelle2014mixture}.
\begin{lem}\label{lem_bracketingEntropyDecomposition2}
	For all $\delta \in (0,\sqrt{2}]$,
	\begin{align*}
		\cH_{[\cdot],\thell }\left(\delta,S_\bfm\right) \le \cH_{[\cdot],d_{\cP_K} }\left(\frac{\delta}{2},\cP_K\right) + 	\cH_{[\cdot], d_{\gkdb} }\left(\frac{\delta}{2},\gkdb\right),
	\end{align*}
	where
	\begin{align*}
		d^2_{\cP_K}\left(g^+,g^-\right) &= \E{\Yv}{\frac{1}{n}\sum_{i=1}^n d^2_k\left(g^+ \left(\Yv_i\right),g^-(\Yv_i)\right)} = \E{\Yv}{\frac{1}{n}\sum_{i=1}^n \sum_{k=1}^K \left(\sqrt{g^+_k \left(\Yv_i\right)}-\sqrt{g^-_k \left(\Yv_i\right)}\right)^2},\\
		d^2_{\gkdb}\left(\Phi^+,\Phi^-\right) &= \E{\Yv}{\frac{1}{n}\sum_{i=1}^n \sum_{k=1}^K d^2_\xv\left(\Phi^+_k \left(\cdot,\Yv_i\right),\Phi^-_k\left(\cdot,\Yv_i\right)\right)}
		\nn\\
		&= \E{\Yv}{\frac{1}{n}\sum_{i=1}^n \sum_{k=1}^K \int_{\cX}\left(\sqrt{\Phi^+_k \left(\xv,\Yv_i\right)}-\sqrt{\Phi^+_k \left(\xv,\Yv_i\right)}\right)^2 d\xv}.
	\end{align*}
\end{lem}
Next, we make use of \cref{lem_Inequality_Hellinger_Supx_dy_dk,}, which is proved in \cref{sec.lem_Inequality_Hellinger_Supx_dy_dk}, to provide an upper bound on the bracketing entropy of $S_\bfm$ ($\cP_K$) on distances $\thell$ ($d_{\cP_K}$), respectively.
\begin{lem}\label{lem_Inequality_Hellinger_Supx_dy_dk}
	It holds that
	\begin{align}
		\thell(s,t) &\le \sup_\yv d_\xv(s,t) \text{, and }
		\cH_{[\cdot],\thell}\left(\delta,S_\bfm\right) \le \cH_{[\cdot],\sup_\yv d_\xv }\left(\delta,S_\bfm\right), \label{eq_Inequality_Hellinger_Supx_dy}\\
		d_{\cP_K}\left(g^+,g^-\right)&\le \sup_\yv d_k(g^+,g^-) \text{, and } \cH_{[\cdot],d_{\cP_K} }\left(\frac{\delta}{2},\cP_K\right) \le \cH_{[\cdot],\sup_\yv d_k }\left(\frac{\delta}{2},\cP_K\right). \label{eq_Inequality_Hellinger_Supx_dk}
	\end{align}
\end{lem}


\cref{lem_bracketingEntropyDecomposition2,lem_Inequality_Hellinger_Supx_dy_dk} imply that
\begin{align*}
	\cH_{[\cdot],\thell}\left(\delta,S_\bfm\right) \le \cH_{[\cdot],\sup_\yv d_k }\left(\frac{\delta}{2},\cP_K\right) + 	\cH_{[\cdot],d_{\gkdb} }\left(\frac{\delta}{2},\gkdb\right).
\end{align*}

We next define the metric entropy of the set $\cW_K$: $\entropy(\delta,\cW_K)$, which measures the logarithm of the minimal number of balls of radius at most $\delta$, according to a distance $d_{\norm{\sup}_\infty}$, needed to cover $\cW_K$,  where
\begin{align} \label{eq.def.metricEntropy-Distance}
	d_{\norm{\sup}_\infty} \left(\left(\bfs_k\right)_{k\in[K]},\left(\bft_k\right)_{k\in[K]}\right) = \max_{k \in [K]}\sup_{\xv \in \cX} \norm{\bfs_k(\xv) - \bft_k(\xv)}_2,
\end{align}
for any $K$-tuples of functions $\left(\bfs_k\right)_{k\in[K]}$ and $\left(\bft_k\right)_{k\in[K]}$. Here, $\bfs_k,\bft_k:\cX \ni \xv \mapsto \bfs_k(\xv),\bft_k(\xv) \in  \R^p, \forall k \in [K]$, and given $\bfx \in \cX, k\in [K]$, $\norm{\bfs_k(\xv) - \bft_k(\xv)}_2$ is the Euclidean distance in $\R^p$. 

Based on this metric, one can first relate the bracketing entropy of $\cP_{K}$ to $\entropy(\delta,\cW_K)$, and then obtain the upper bound for its entropy via \cref{Extend.montuelle.lemma4new}.
It is worth mentioning that for the Gaussian gating parameters, the technique for handling the logistic weights of \cite{montuelle2014mixture} is not directly applicable to the BLoMPE setting. Therefore, by using the previous reparameterization trick, \citet[Lemmas 5.4 and 5.8]{nguyen2021nonGLoME} allow for the control of the metric entropy of the parameters of Gaussian gating functions.

\begin{lem}[Lemmas 5.5 from \cite{nguyen2021nonGLoME} ]\label{Extend.montuelle.lemma4new}
	For all $\delta \in (0,\sqrt{2}]$,
	\begin{align*}
		\cH_{[\cdot],\sup_\yv d_k}\left(\frac{\delta}{2},\cP_K\right) 
		\le \entropy\left(\frac{ 3 \sqrt{3}\delta}{8 \sqrt{K}},\cW_K\right) \le \dim\left(\cW_K\right) \left(C_\cW + \ln \left(\frac{ 8 \sqrt{K}}{3 \sqrt{3}\delta}\right)\right),
	\end{align*}
	where  $C_\cY := \sup_{\yv \in \cY} \left\|\yv\right\|_\infty < \infty$ whenever $\cY$ is bounded, $\cU:= \cY \times \cY \times \left[a_\Gammab,A_\Gammab\right]^{L^2}$, 
	\begin{align*}
		C_\cW &:= \frac{1}{\dim\left(\cW_K\right)}\ln C_0, C_0 := \left(6C_\cb C_\cY L\right)^{KL}\left(6C_\Gammab A_\Gammab L^2\right)^{\frac{L(L+1)}{2}K} \left(\frac{3}{a_\pib}\right)^{K-1} K \left(2 \pi e\right)^{K/2},\\
		0 < \left(C_{\cb}\right)_{1,\ldots,L}^\top&:=\max_{k\in[K]}\sup_{ \left(\yv,\bfc_k,\vect\left(\Gammab_k\right)\right)\in \cU } \left|\nabla_{\bfc_k} \ln\left|\Phi_L(\yv;\bfc_k,\Gammab_k)\right|\right| < \infty,\\
		0 < \left(C_{\Sigmab}\right)_{1,\ldots,L^2}^\top&:=\max_{k\in[K]}\sup_{\left(\yv,\bfc_k,\vect\left(\Gammab_k\right)\right) \in \cU} \left|\nabla_{\vect\left(\Gammab_k\right)}\ln\left| \Phi_L(\yv;\bfc_k,\Gammab_k)\right|\right| < \infty,
	\end{align*}
	and $\vect(\cdot)$ denotes the
	vectorization operator that stacks the columns of a matrix into a vector.
	
\end{lem}
\cref{lem.bracketingEntropyGaussianBlock} allows us to construct the Gaussian brackets to handle the metric entropy for Gaussian experts, which is established in \cref{proof.lem.bracketingEntropyGaussianBlock}.
\begin{lem}\label{lem.bracketingEntropyGaussianBlock}
	\begin{align}\label{eq.bracketingEntropyGaussianBlock}
		\cH_{[\cdot],d_{\gkdb} }\left(\frac{\delta}{2},\gkdb\right) \le \dim\left(\gkdb\right) \left(C_{\gkdb} + \ln \left(\frac{1}{\delta}\right)\right).
	\end{align}
\end{lem}

Finally, \eqref{eq.bractketingEntropyModelSm} is proved via \cref{Extend.montuelle.lemma4new,lem.bracketingEntropyGaussianBlock}.
Indeed, with the fact that $\dim(S_\bfm) =\dim(\cW_K) + \dim\left(\gkdb\right)$, it follows
\begin{align*}
	&\cH_{[\cdot],\thell}\left(\delta,S_\bfm\right)\nn\\
	 &\le \cH_{[\cdot],\sup_\yv d_k}\left(\frac{\delta}{2},\cP_K\right) + \cH_{[\cdot],d_{\gkdb} }\left(\frac{\delta}{2},\gkdb\right)
	\nn\\
	&\le \dim\left(\cW_K\right) \left(C_\cW + \ln \left(\frac{8\sqrt{K}}{3 \sqrt{3}\delta}\right)\right) + \dim\left(\gkdb\right) \left(C_{\gkdb} + \ln \left(\frac{1}{\delta}\right)\right)\nn\\
	%
	%
	&= \dim\left(S_\bfm\right) \left[\frac{\dim\left(\cW_K\right)}{\dim\left(S_\bfm\right)} \left(C_\cW + \ln \left(\frac{8\sqrt{K}}{3 \sqrt{3}} \right)+\ln \left(\frac{1}{\delta}\right)\right) +  \frac{\dim\left(\gkdb\right)}{\dim\left(S_\bfm\right)} \left(C_{\gkdb} + \ln \left(\frac{1}{\delta}\right)\right)\right]\nn\\
	%
	%
	&= \dim(S_\bfm)\left(C_\bfm + \ln \left(\frac{1}{\delta}\right)\right), \text{ where }\nn\\
	%
	C_\bfm &= \frac{\dim(\cW_K)}{\dim(S_\bfm)} \left(C_\cW + \ln \left(\frac{8\sqrt{K}}{3 \sqrt{3}}\right)\right)+ \frac{\dim\left(\gkdb\right) C_{\gkdb}}{\dim\left(S_\bfm\right)}
	\nn\\&
	\le C_\cW + \ln \left(\frac{8\sqrt{K_{\max}}}{3\sqrt{3}}\right) + C_{\gkdb} :=\fC.
\end{align*}
It is interesting that the constant $\fC$ does not depend on the dimension $\dim\left(S_\bfm\right)$ of the model, thanks to the hypothesis that $C_\cW$ is common for every model $S_\bfm$ in the collection. Therefore, \cref{prop.DudleyClassicalEntropy} implies that, give $C = 2 \left(\sqrt{\fC} + \sqrt{\pi}\right)^2$, the model complexity $\cD_\bfm$ satisfies
\begin{align*}
	\cD_\bfm \equiv n \delta^2_\bfm \le \dim(S_\bfm) \left(2 \left(\sqrt{\fC} + \sqrt{\pi}\right)^2 + \left(\ln \frac{n}{\left(\sqrt{\fC} + \sqrt{\pi}\right)^2\dim\left(S_\bfm\right)}\right)_+\right)
	\le \dim(S_\bfm) \left(C+\ln n\right).
\end{align*}

To this end, \cref{thm_5_1_devijver2015finite} implies that when a collection of BLoMPE models $ \left(S_\bfm\right)_{\bfm \in \cM}$ with the penalty functions satisfies $\pen(\bfm) \geq \kappa \left[\dim(S_\bfm) \left(C+\ln n\right)+ (1 \vee \tau)\xi_\bfm\right]$ with $\kappa > \kappa_0$,
the oracle inequality in \cref{weakOracleInequality} holds.

\section*{Acknowledgments}
TrungTin Nguyen is supported by a ``Contrat doctoral'' from the French Ministry of Higher Education
and Research. Faicel Chamroukhi is granted by the French National Research Agency (ANR) grant \href{https://anr.fr/en/funded-projects-and-impact/funded-projects/project/funded/project/b2d9d3668f92a3b9fbbf7866072501ef-f004f5ad27/?tx_anrprojects\_funded\%5Bcontroller\%5D=Funded&cHash=895d4c3a16ad6a0902e6515eb65fba37}{SMILES ANR-18-CE40-0014}. Hien Duy Nguyen is funded by Australian Research Council grant number DP180101192. This research is funded directly by the Inria \href{https://team.inria.fr/statify/projects/lander/}{LANDER} project. TrungTin Nguyen also sincerely acknowledges Inria Grenoble-Rhône-Alpes Research Centre for a valuable Visiting PhD Fellowship working with \href{https://team.inria.fr/statify/}{STATIFY} team so that this research can be completed, 
Emilie Devijver for fruitful statistical discussions.


\appendix

\section*{Appendix}

\section{Lemma proofs}\label{proofLemma}

\subsection{Proof of Lemma \ref{lem_Inequality_Hellinger_Supx_dy_dk}} \label{sec.lem_Inequality_Hellinger_Supx_dy_dk}
We first aim to prove that $\thel(s,t) \le \sup_\yv d^2_\xv(s,t)$. Indeed, by definition, it follows that
\begin{align*}
	\thel\left(s,t\right) &= \E{\Yv}{\frac{1}{n} \sum_{i=1}^n d^2_\xv\left(s\left(\cdot|\Yv_i\right),t\left(\cdot|\Yv_i\right)\right)}
	=\frac{1}{n} \sum_{i=1}^n \E{\Yv}{d^2_\xv\left(s\left(\cdot|\Yv_i\right),t\left(\cdot|\Yv_i\right)\right)}\nn\\
	&=\frac{1}{n} \sum_{i=1}^n \int_{\cY}d^2_\xv\left(s\left(\cdot|\yv\right),t\left(\cdot|\yv\right)\right)s_{\xv,0}(\yv)d\yv \le  \sup_\yv d^2_\xv~(s,t) \frac{1}{n} \sum_{i=1}^n\int_{\cY}s_{\xv,0}(\yv)d\yv = \sup_\yv d^2_\xv~(s,t),
\end{align*}
where $s_{\xv,0}$ denotes that marginal PDF of $s_0$, \wrt~$\xv$.
Consequently, it holds that $\thell(s,t) = \sqrt{\thel(s,t)} \le \sqrt{\sup_\yv d^2_\xv(s,t)}= \sup_\yv d_\xv(s,t)$.
To prove that  $$\cH_{[\cdot],\thell}\left(\delta,S_\bfm\right) \le \cH_{[\cdot],\sup_\yv d_\xv }\left(\delta,S_\bfm\right),$$ 
it is sufficient to check that $$\cN_{[\cdot],\thell}\left(\delta,S_\bfm\right) \le \cN_{[\cdot],\sup_\yv d_\xv }\left(\delta,S_\bfm\right).$$
By using the definition of bracketing entropy in \eqref{eq_definition_BracketingEntropy} and $\thell(s,t) \le \sup_\yv d_\xv(s,t)$, given
\begin{align*}
	A &= \left\{n \in \Ns: \exists t^-_1,t^+_1,\ldots,t^-_n,t^+_n \text{ s.t }\sup_\yv d_\xv(s,t)\left(t^-_k,t^+_k\right) \le \delta,S_\bfm \subset \bigcup_{k=1}^n \left[t^-_k,t^+_k\right] \right\},\\
	B &=\left\{n\in \Ns: \exists t^-_1,t^+_1,\ldots,t^-_n,t^+_n \text{ s.t }\thell\left(t^-_k,t^+_k\right) \le \delta,S_\bfm \subset \bigcup_{k=1}^n \left[t^-_k,t^+_k\right] \right\},
\end{align*}
it leads to that $A \subset B$ and then \eqref{eq_Inequality_Hellinger_Supx_dy} follows, since
\begin{align*}
	\cN_{[\cdot],\sup_\yv d_\xv(s,t)}\left(\delta,S_\bfm\right)= \min A\ge \min B =\cN_{[\cdot],\thell}\left(\delta,S_\bfm\right).
\end{align*}

With the similar argument as in the proof of \eqref{eq_Inequality_Hellinger_Supx_dy}, it holds that $d_{\cP_K}\left(g^+,g^-\right) \le \sup_\yv d_k(g^+,g^-)$ and \eqref{eq_Inequality_Hellinger_Supx_dk} is proved.
\subsection{Proof of Lemma \ref{lem.bracketingEntropyGaussianBlock}}\label{proof.lem.bracketingEntropyGaussianBlock}
It is worth mentioning that without any structures on covariance matrices of Gaussian experts from the collection $\cM$, Lemma \ref{lem.bracketingEntropyGaussianBlock} can be proved using Proposition 2 from \cite{montuelle2014mixture} and \citet[Appendix B.2.3]{montuelle2014mixture}, for constructing of Gaussian brackets to deal with the Gaussian experts. However, dealing with block-diagonal covariance matrices with random subcollection is much more challenging. We have to establish more constructive bracketing entropies in the spirits of \cite{maugis2011non,devijver2015finite,devijver2018block}. 

Given any $k \in [K]$, by defining
\begin{align} \label{eq_define_GaussianExpert_1D}
	\gdbk &= \left\{\cX \times\cY \ni \left(\xv,\yv\right)\mapsto \Phi\left(\xv;\upsilonb_{k,d}(\yv),\Sigmab_k\left(\bfB_k\right)\right)=:\Phi_k: \upsilonb_{k,d} \in \Upsilondk, \Sigmab_k\left(\bfB_k\right) \in \bfV_k(\bfB_k) \right\},
\end{align}
it follows that $\gkdb = \prod_{k=1}^K \gdbk$, where $\prod$ stands for the cartesian product. By using \cref{lem_HGK_HG}, which is proved in \cref{sec.prooflem_HGK_HG}, it follows that
\begin{align} \label{eq_HGK_HG}
	\cH_{[\cdot],d_{\gkdb} }\left(\frac{\delta}{2},\gkdb\right) \le \sum_{k=1}^K \cH_{[\cdot],d_{\gdbk} }\left(\frac{\delta}{2\sqrt{K}},\gdbk\right).
\end{align}
\begin{lem} \label{lem_HGK_HG}
	Given $\gkdb = \prod_{k=1}^K \gdbk$, where $\gdbk$ is defined in \eqref{eq_define_GaussianExpert_1D}, it holds that
	\begin{align*}
		\cN_{[\cdot],d_{\gkdb}}\left(\frac{\delta}{2},\gkdb\right) \le \prod_{k=1}^K\cN_{[\cdot],d_{\gdbk}}\left(\frac{\delta}{2\sqrt{K}},\gdbk\right),
	\end{align*}
	where for any $\Phi^+,\Phi^- \in \gkdb$ and any $\Phi_k^+,\Phi_k^- \in \gdbk,k\in[K]$,
	\begin{align*}
		d^2_{\gkdb}\left(\Phi^+,\Phi^-\right) &= \E{\Yv}{\frac{1}{n}\sum_{i=1}^n \sum_{k=1}^K \hel\left(\Phi^+_k \left(\cdot,\Yv_i\right),\Phi^-_k\left(\cdot,\Yv_i\right)\right)},\nn\\
		d^2_{\gdbk}\left(\Phi_k^+,\Phi_k^-\right) &= \E{\Yv}{\frac{1}{n}\sum_{i=1}^n  \hel\left(\Phi_k^+ \left(\cdot,\Yv_i\right),\Phi_k^-\left(\cdot,\Yv_i\right)\right)}
		.
	\end{align*}
\end{lem}
\cref{lem.bracketingEntropyGaussianBlock} is proved via \eqref{eq_HGK_HG} and \cref{lem.bracketingEntropyGaussianBlock1-D}, which is proved in \cref{sec.lem.bracketingEntropyGaussianBlock1-D}.
\begin{lem}\label{lem.bracketingEntropyGaussianBlock1-D}
	By defining $\gdbk$ as in \eqref{eq_define_GaussianExpert_1D}, for all $\delta \in (0,\sqrt{2}]$, it holds that
	\begin{align}
		\cH_{[\cdot],d_{\gdbk} }\left(\frac{\delta}{2},\gdbk\right) &\le \dim\left(\gdbk\right) \left(C_{\gdbk} + \ln \left(\frac{1}{\delta}\right)\right),\label{eq_bracketingEntropyGaussianBlock1_D} \text{ where }\\
		D_{\bfB_k} &=  \sum_{g=1}^{G_k} \frac{\card\left(d_k^{[g]}\right)\left(\card\left(d_k^{[g]}\right)-1\right)}{2},\nn\\
		C_{\gdbk} &= \frac{D_{\bfB_k} \ln \left(\frac{6 \sqrt{6} \lambda_M D^2\left(D-1\right)}{ \lambda_m D_{\bfB_k}}\right) + \dim\left(\Upsilondk\right) \ln\left(\frac{6 \sqrt{2D}\exp\left(C_{\Upsilondk}\right)}{ \sqrt{\lambda_m}}\right)}{\dim\left(\gdbk\right)}\nn.
	\end{align}
\end{lem}
Indeed, \eqref{eq_HGK_HG} and \eqref{eq_bracketingEntropyGaussianBlock1_D} lead to
\begin{align*}
	\cH_{[\cdot],d_{\gkdb} }\left(\frac{\delta}{2},\gkdb\right) &\le \sum_{k=1}^K \cH_{[\cdot],d_{\gdbk} }\left(\frac{\delta}{2\sqrt{K}},\gdbk\right)\nn\\
	&\le  \sum_{k=1}^K  \dim\left(\gdbk\right)\left(C_{\gdbk} + \ln \left(\sqrt{K}\right) + \ln \left(\frac{1}{\delta}\right)\right)\\
	& \le \dim\left(\gkdb\right) \left(C_{\gkdb} + \ln \left(\frac{1}{\delta}\right)\right).
\end{align*}
Here, $C_{\gkdb} = \sum_{k=1}^K C_{\gdbk} + \ln \left(\sqrt{K}\right)$ and note that $\dim\left(\gkdb\right) = \sum_{k=1}^K  \dim\left(\gdbk\right)$, $\dim\left(\gdbk\right) = D_{\bfB_k}+\dim\left(\Upsilondk\right)$, $\dim\left(\Upsilondk\right) = D d_{\Upsilondk}, C_{\Upsilondk}=\sqrt{D}d_{\Upsilondk}T_{\Upsilondk}$ (in cases where linear combination of bounded functions are used for means, \ie $\Upsilondk = \Upsilonb_b$) or $\dim\left(\Upsilondk\right) = D \binom{d_{\Upsilondk} + L}{ L}$, $C_{\Upsilondk} = \sqrt{D} \binom{d_{\Upsilondk} + L}{L} T_{\Upsilondk}$ (in cases where we use polynomial means, \ie $\Upsilondk = \Upsilonb_p$).

\subsubsection{Proof of Lemma \ref{lem_HGK_HG}}\label{sec.prooflem_HGK_HG}
By the definition of the bracketing entropy in \eqref{eq_definition_BracketingEntropy}, for each $k\in[K]$, let $\left\{\left[\Phi^{l,-}_k,\Phi^{l,+}_k\right]\right\}_{ 1 \le  l \le  \cN_{\gdbk}}$ be a minimal covering of $\delta_k$ brackets for $d_{\gdbk}$ of $\gdbk$, with cardinality $
\cN_{\gdbk}$. This leads to
\begin{align*}
	\forall l \in \left[\cN_{\gdbk}\right], d_{\gdbk}\left(\Phi^{l,-}_k,\Phi^{l,+}_k\right) \le  \delta_k.
\end{align*} 
Therefore, we claim that the set $\left\{\prod_{k=1}^K\left[\Phi^{l,-}_k,\Phi^{l,+}_k\right]\right\}_{ 1 \le  l \le  \cN_{\gdbk}}$ is a covering of $\frac{\delta}{2}$-bracket for $d_{\gkdb}$ of $\gkdb$ with cardinality $\prod_{k=1}^K \cN_{[\cdot], d_{\gdbk} }\left(\delta_k,\gdbk\right)$.
Indeed, let any $\Phi = \left(\Phi_k\right)_{k\in[K]} \in \gkdb$. Consequently, for each $k\in[K], \Phi_k \in \gdbk$, there exists $l(k)\in\left[ \cN_{\gdbk}\right]$, such that
\begin{align*}
	\Phi^{l(k),-}_k\le  \Phi_k \le \Phi^{l(k),+}_k,
	d^2_{\gdbk}\left(\Phi^{l(k),+}_k,\Phi^{l(k),-}_k\right) \le \left(\delta_k\right)^2.
\end{align*}
Then, it follows that $\Phi \in \left[\Phi^{-},\Phi^{+}\right] \in \left\{\prod_{k=1}^K\left[\Phi^{l,-}_k,\Phi^{l,+}_k\right]\right\}_{ 1 \le  l \le  \cN_{\gdbk}}$, with $\Phi^{-} = \left(\Phi^{l(k),-}_k\right)_{k\in[K]},\Phi^{+} = \left(\Phi^{l(k),+}_k\right)_{k\in[K]}$, which implies that  $\left\{\prod_{k=1}^K\left[\Phi^{l,-}_k,\Phi^{l,+}_k\right]\right\}_{ 1 \le  l \le  \cN_{\gdbk}}$ is a bracket covering of $\gkdb$.

Now, we want to verify that the size of this bracket is $\delta/2$ by choosing $\delta_k = \frac{\delta}{2\sqrt{K}}, \forall k \in [K]$. It follows that 
\begin{align*}
	d^2_{\gkdb}\left(\Phi^{-},\Phi^{+}\right) &= \E{\Yv}{\frac{1}{n}\sum_{i=1}^n \sum_{k=1}^K \hel\left(\Phi^{l(k),-}_k \left(\cdot,\Yv_i\right),\Phi^{l(k),+}_k\left(\cdot,\Yv_i\right)\right)}\nn\\
	&= \sum_{k=1}^K\E{\Yv}{\frac{1}{n}\sum_{i=1}^n  \hel\left(\Phi^{l(k),-}_k \left(\cdot,\Yv_i\right),\Phi^{l(k),+}_k\left(\cdot,\Yv_i\right)\right)} \nn\\
	&= \sum_{k=1}^K d^2_{\gdbk}\left(\Phi^{l(k),-}_k ,\Phi^{l(k),+}_k\right)
	\le K \left(\frac{\delta}{2\sqrt{K}}\right)^2 = \left(\frac{\delta}{2}\right)^2.
\end{align*}
To this end, by definition of a minimal $\frac{\delta}{2}$-bracket covering number for $\gkdb$, \cref{lem_HGK_HG} is proved.
%

\subsubsection{Proof of Lemma \ref{lem.bracketingEntropyGaussianBlock1-D}}\label{sec.lem.bracketingEntropyGaussianBlock1-D}

To provide the upper bound of
the bracketing entropy in  \eqref{eq_bracketingEntropyGaussianBlock1_D}, our technique is adapted from the work of \cite{genovese2000rates} for unidimensional Gaussian mixture families, which is then generalized to multidimensional case by \cite{maugis2011non}. Furthermore, we make use of the results from \cite{devijver2018block} to deal with block-diagonal covariance matrices, $\bfV_k\left(\bfB_k\right),k\in[K]$, and from \cite{montuelle2014mixture} to handle the means of Gaussian experts $\Upsilondk,k\in[K]$.
The main idea is to define firstly a net over the parameter spaces of Gaussian experts, $\Upsilondk\times \bfV_k\left(\bfB_k\right), k\in[K]$, and to construct a bracket covering of $\gdbk$ according to the tensorized Hellinger distance. 
Note that $\dim\left(\gdbk\right)=\dim\left(\Upsilondk\right)+\dim\left(\bfV_k\left(\bfB_k\right)\right)$. 

\paragraph{Step 1: Construction of a net for the block-diagonal covariance matrices.}
Firstly, for $k\in[K]$, we denote by $\adj\left(\Sigmab_k\left(\bfB_k\right)\right)$ the adjacency matrix associated to the covariance matrix $\Sigmab_k\left(\bfB_k\right)$. Note that this matrix of size $D^2$ can be defined by a vector of concatenated upper triangular vectors. We are going to make use of the result from \cite{devijver2018block} to handle the block-diagonal covariance matrices $\Sigmab_k\left(\bfB_k\right)$, via its corresponding adjacency matrix. To do this, we need to construct a discrete space for $\left\{0,1\right\}^{D(D-1)/2}$, which is a one-to-one correspondence (bijection) with
$$\cA_{\bfB_k} = \left\{\bfA_{\bfB_k} \in \cS_D\left(\left\{0,1\right\}\right):\exists \Sigmab_k\left(\bfB_k\right) \in \bfV_k\left(\bfB_k\right) \text{ s.t } \adj\left(\Sigmab_k\left(\bfB_k\right)\right) = \bfA_{\bfB_k} \right\},$$
where $\cS_D\left(\left\{0,1\right\}\right)$ is the set of symmetric matrices of size $D$ taking values on $\left\{0,1\right\}$. 

Then, we want to deduce a discretization of the set of covariance matrices. Let $h$ denotes Hamming distance on $\left\{0,1\right\}^{D(D-1)/2}$ defined by
$$d(z,z') = \sum_{i=1}^n \Indi\left\{z \neq z'\right\}, \text{ for all } z,z' \in \left\{0,1\right\}^{D(D-1)/2}.$$
Let $\left\{0,1\right\}_{\bfB_k}^{D(D-1)/2}$ be the subset of $\left\{0,1\right\}^{D(D-1)/2}$ of vectors for which the corresponding graph has structure $\bfB_k = \left(d^{[g]}_k\right)_{g\in\left[G_k\right]}$. Corollary 1 and Proposition 2 from Supplementary Material A of \cite{devijver2018block} imply that there exists some subset $\cR$ of $\left\{0,1\right\}^{D(D-1)/2}$, as well as its equivalent $\cA^{\disc}_{\bfB_k}$ for adjacency matrices such that, given $\epsilon > 0$, and
\begin{align*}
	{\tilde{S}}^{\disc}_{\bfB_k}(\epsilon) = \left\{\Sigmab_k\left(\bfB_k\right) \in \cS_D^{++}(\mathbb{R}): \adj\left(\Sigmab_k\left(\bfB_k\right)\right) \in \cA^{\disc}_{\bfB_k}, \left[\Sigmab_k\left(\bfB_k\right)\right]_{i,j} = \sigma_{i,j}\epsilon,\sigma_{i,j} \in \left[\frac{-\lambda_M}{\epsilon},\frac{\lambda_M}{\epsilon}\right]\bigcap \Z \right\},
\end{align*} 
it holds that 
\begin{align}
	\left\|\Sigmab_k\left(\bfB_k\right)-\widetilde{\Sigmab}_k\left(\bfB_k\right)\right\|_2^2 &\le \frac{D_{\bfB_k}}{2} \wedge \epsilon^2,\forall \left(\Sigmab_k\left(\bfB_k\right),\widetilde{\Sigmab}_k\left(\bfB_k\right) \right)\in \left({\tilde{S}}^{\disc}_{\bfB_k}(\epsilon)\right)^2 \text{\st} \Sigmab_k\left(\bfB_k\right) \neq \widetilde{\Sigmab}_k\left(\bfB_k\right),\nn\\
	\card\left({\tilde{S}}^{\disc}_{\bfB_k}(\epsilon)\right) & \le  \left(\Bigg\lfloor \frac{2\lambda_M}{\epsilon}\Bigg\rfloor\frac{D\left(D-1\right)}{2 D_{\bfB_k}}\right)^{D_{\bfB_k}},\\
	D_{\bfB_k} &=\dim\left(\bfV_k\left(\bfB_k\right)\right)= \sum_{g=1}^{G_k} \frac{\card\left(d_k^{[g]}\right)\left(\card\left(d_k^{[g]}\right)-1\right)}{2}. \label{eq.upperCardCovariance}
\end{align}
By choosing $\epsilon^2 \le \frac{D_{\bfB_k}}{2}$, given $\Sigmab_k\left(\bfB_k\right) \in \bfV_k\left(\bfB_k\right)$, then there exists $\widetilde{\Sigmab}_k\left(\bfB_k\right) \in {\tilde{S}}^{\disc}_{\bfB_k}(\epsilon)$, such that 
\begin{align}
	\left\|\Sigmab_k\left(\bfB_k\right)-\widetilde{\Sigmab}_k\left(\bfB_k\right)\right\|_2^2 \le \epsilon^2. \label{eq.netsCovarianceG1}
\end{align}

\paragraph{Step 2: Construction of a net for the mean functions.}
Based on $\widetilde{\Sigmab}_k\left(\bfB_k\right)$, we can construct the following bracket covering of $\gdbk$ by defining the nets for the means of Gaussian experts.
The proof of Lemma 1, page 1693, from \cite{montuelle2014mixture} implies that
\begin{align*}
	\cN_{\left[\cdot\right],\sup_\yv \left\|\cdot\right\|_2}\left(\delta_{\Upsilondk},\Upsilondk\right) \le \left(\frac{\exp\left(C_{\Upsilondk}\right)}{\delta_{\Upsilondk}}\right)^{\dim\left(\Upsilondk\right)}.
\end{align*} 
Here $\dim\left(\Upsilondk\right) = D d_{\Upsilondk}$, and $C_{\Upsilondk}=\sqrt{D}d_{\Upsilondk}T_{\Upsilondk}$ in the general case or
$\dim\left(\Upsilondk\right) = D \binom{d_{\Upsilondk} + L}{ L}$, and $C_{\Upsilondk} = \sqrt{D} \binom{d_{\Upsilondk} + L}{L} T_{\Upsilondk}$ in the special case of polynomial means. Then, by the definition of bracketing entropy in \eqref{eq_definition_BracketingEntropy}, for any minimal $\delta_{\Upsilondk}$-bracketing covering of the means from Gaussian experts, denoted by $G_{\Upsilondk}\left(\delta_{\Upsilondk}\right)$, it is true that
\begin{align}
	\card\left(G_{\Upsilondk}\left(\delta_{\Upsilondk}\right)\right) \le \left(\frac{\exp\left(C_{\Upsilondk}\right)}{\delta_{\Upsilondk}}\right)^{\dim\left(\Upsilondk\right)}.\label{eq.cardMeanGaussianEx}
\end{align}
Therefore, given $\alpha >0$, which is specified later, we claim that the set
\begin{align*}
	\left\{\left[l,u\right]\left| 
	\begin{array}{l} 
		l(\xv,\yv) = \left(1+2\alpha\right)^{-D} \Phi\left(\xv;\widetilde{\upsilonb}_{k,d}(\yv),\left(1+\alpha\right)^{-1}\widetilde{\Sigmab}_k\left(\bfB_k\right)\right), \\
		u(\xv,\yv) = \left(1+2\alpha\right)^{D} \Phi\left(\xv;\widetilde{\upsilonb}_{k,d}(\yv),\left(1+\alpha\right)\widetilde{\Sigmab}_k\left(\bfB_k\right)\right), \\
		\widetilde{\upsilonb}_{k,d} \in G_{\Upsilondk}\left(\delta_{\Upsilondk}\right),\widetilde{\Sigmab}_k\left(\bfB_k\right) \in {\tilde{S}}^{\disc}_{\bfB_k}(\epsilon)
	\end{array}
	\right.\right\},
\end{align*}
is a $\delta_{\Upsilondk}$-brackets set over $\gdbk$. Indeed, let $\cX\times\cY \ni  (\xv,\yv) \mapsto f(\xv,\yv) = \Phi\left(\xv;\upsilonb_{k,d}(\yv),\Sigmab_k\left(\bfB_k\right)\right)$ be a function of $\gdbk$, where $\upsilonb_{k,d} \in \Upsilondk$ and $\Sigmab_k\left(\bfB_k\right) \in \bfV_k\left(\bfB_k\right)$. According to \eqref{eq.netsCovarianceG1}, there exists $\widetilde{\Sigmab}_k\left(\bfB_k\right) \in {\tilde{S}}^{\disc}_{\bfB_k}(\epsilon)$, such that 
\begin{align*}
	\left\|\Sigmab_k\left(\bfB_k\right)-\widetilde{\Sigmab}_k\left(\bfB_k\right)\right\|_2^2 \le \epsilon^2.
\end{align*}
By definition of $G_{\Upsilondk}\left(\delta_{\Upsilondk}\right)$, there exists $\widetilde{\upsilonb}_{k,d} \in G_{\Upsilondk}\left(\delta_{\Upsilondk}\right)$, such that 
\begin{align}
	\sup_{\yv \in \cY} \left\|\widetilde{\upsilonb}_{k,d}(\yv)-\upsilonb_{k,d}(\yv)\right\|_2^2 \le \delta^2_{\Upsilondk}. \label{eq.netMeanGaussian}
\end{align}

\paragraph{Step 3: Upper bound of the number of the bracketing entropy.}
Next, we wish to make use of \cref{lem.ratioGaussian} to evaluate the ratio of two Gaussian densities.
\begin{lem}[Proposition C.1 from \cite{maugis2011non}]\label{lem.ratioGaussian}
	Let $\Phi\left(\cdot;\mub_1,\Sigmab_1\right)$ and $\Phi\left(\cdot;\mub_2,\Sigmab_{2}\right)$ be two Gaussian densities. If $\Sigmab_{2}-\Sigmab_1$ is a positive definite matrix then for all $\xv \in \R^D$, 
	\begin{align*}
		\frac{\Phi\left(\xv;\mub_1,\Sigmab_1\right)}{\Phi\left(\xv;\mub_2,\Sigmab_{2}\right)} \le \sqrt{\frac{\left|\Sigmab_{2}\right|}{\left|\Sigmab_1\right|}} \exp\left[\frac{1}{2}\left(\mub_1-\mub_2\right)^\top \left(\Sigmab_{2}-\Sigmab_1\right)^{-1}\left(\mub_1-\mub_2\right)\right].
	\end{align*}
\end{lem}
The following \cref{lem.assumptionPDM} allows us to fulfill the assumptions of \cref{lem.ratioGaussian}.
\begin{lem}[Similar to Lemma B.8 from \cite{maugis2011non}]\label{lem.assumptionPDM}
	Assume that $0 < \epsilon < \lambda^2_m/9$, and set $\alpha = 3 \sqrt{\epsilon} /\lambda_m$.  Then, for every $k\in[K]$, $\left(1+\alpha\right)\widetilde{\Sigmab}_k\left(\bfB_k\right) - \Sigmab_k\left(\bfB_k\right)$ and $\Sigmab_k\left(\bfB_k\right) - \left(1+\alpha\right)^{-1}\widetilde{\Sigmab}_k\left(\bfB_k\right)$ are both positive definite matrices. Moreover, for all $\xv \in \R^D$,
	\begin{align*}
		\xv^\top \left[\left(1+\alpha\right)\widetilde{\Sigmab}_k\left(\bfB_k\right) - \Sigmab_k\left(\bfB_k\right)\right]\xv \ge \epsilon \left\|\xv\right\|_2^2,\quad
		\xv^\top \left[\Sigmab_k\left(\bfB_k\right) - \left(1+\alpha\right)^{-1}\widetilde{\Sigmab}_k\left(\bfB_k\right)\right]\xv \ge \epsilon \left\|\xv\right\|_2^2.
	\end{align*} 	
\end{lem}
\begin{proof}[Proof of \cref{lem.assumptionPDM}]
	For all $\xv \neq \zero$, since $\sup_{\lambda \in \text{vp}\left(\Sigmab_k\left(\bfB_k\right)-\widetilde{\Sigmab}_k\left(\bfB_k\right)\right)}\left|\lambda\right|=\left\|\Sigmab_k\left(\bfB_k\right)-\widetilde{\Sigmab}_k\left(\bfB_k\right)\right\|_2 \le \epsilon$, where $\text{vp}$ denotes the spectrum of matrix, $-\epsilon \ge -\lambda_m/3$, and $\alpha = 3 \epsilon /\lambda_m$, it follow that
	\begin{align*}
		\xv^\top \left[\left(1+\alpha\right)\widetilde{\Sigmab}_k\left(\bfB_k\right) - \Sigmab_k\left(\bfB_k\right)\right]\xv &=\left(1+\alpha\right)\xv^\top \left[\widetilde{\Sigmab}_k\left(\bfB_k\right) - \Sigmab_k\left(\bfB_k\right)\right]\xv+\alpha \xv^\top  \Sigmab_k\left(\bfB_k\right)\xv\\
		&\ge -\left(1+\alpha\right)\left\|\widetilde{\Sigmab}_k\left(\bfB_k\right) - \Sigmab_k\left(\bfB_k\right)\right\|_2\left\|\xv\right\|_2^2 + \alpha \lambda_m \left\|\xv\right\|_2^2\\ 
		&\ge \left(\alpha \lambda_m-\left(1+\alpha\right)\epsilon\right)\left\|\xv\right\|_2^2 =  \left(\alpha \lambda_m-\alpha\epsilon-\epsilon\right)\left\|\xv\right\|_2^2\\
		& \ge \left(\frac{2}{3}\alpha \lambda_m-\epsilon\right)\left\|\xv\right\|_2^2= \epsilon\left\|\xv\right\|_2^2>0, \text{ and}\\
		\xv^\top \left[\Sigmab_k\left(\bfB_k\right) - \left(1+\alpha\right)^{-1}\widetilde{\Sigmab}_k\left(\bfB_k\right)\right]\xv &=\left(1+\alpha\right)^{-1}\xv^\top \left[ \Sigmab_k\left(\bfB_k\right)-\widetilde{\Sigmab}_k\left(\bfB_k\right)\right]\xv+\left(1-\left(1+\alpha\right)^{-1}\right) \xv^\top  \Sigmab_k\left(\bfB_k\right)\xv\\
		&\ge \left(\frac{\alpha \lambda_m-\epsilon}{1+\alpha}\right)\left\|\xv\right\|_2^2
		= \frac{2\epsilon}{1+\alpha}\left\|\xv\right\|_2^2
		\ge \epsilon\left\|\xv\right\|_2^2>0 \left(\text{ since }0 < \alpha < 1\right).
	\end{align*}
\end{proof}
By \cref{lem.ratioGaussian} and the same argument as in the proof of Lemma B.9 from \cite{maugis2011non}, given  $0 < \epsilon < \lambda_m/3$, where $\epsilon$ is chosen later, and $\alpha = 3 \epsilon /\lambda_m$, we obtain
\begin{align}
	\max\left\{\frac{l(\xv,\yv)}{f(\xv,\yv)},\frac{f(\xv,\yv)}{u(\xv,\yv)}\right\} \le \left(1+2\alpha\right)^{-\frac{D}{2}} \exp\left(\frac{\left\|\upsilonb_{k,d}(\yv)-\widetilde{\upsilonb}_{k,d}(\yv)\right\|_2^2}{2\epsilon}\right).\label{eq.conditionRadiusDiscVariance}
\end{align}
Because $\ln\left(\cdot\right)$ is a non-decreasing function, $\ln\left(1+2\alpha\right) \ge \alpha, \forall \alpha \in \left[0,1\right]$. Combined with \eqref{eq.netMeanGaussian} where $\delta^2_{\Upsilondk} = D \alpha \epsilon$, we conclude that
\begin{align*}
	\max\left\{\ln \left(\frac{l(\xv,\yv)}{f(\xv,\yv)}\right),\ln\left(\frac{f(\xv,\yv)}{u(\xv,\yv)}\right)\right\} &\le -\frac{D}{2} \ln \left(1+2\alpha\right)+ \frac{\delta^2_{\Upsilondk}}{2\epsilon}\le -\frac{D}{2} \alpha+ \frac{\delta^2_{\Upsilondk}}{2\epsilon}=0.
\end{align*}
This means that $l(\xv,\yv) \le f(\xv,\yv) \le u(\xv,\yv), \forall (\xv,\yv) \in \cX \times \cY$. 
Hence, it remains to bound the size of bracket $\left[l,u\right]$ \wrt~$d_{\gdbk}$.
To this end, we aim to verify that $d^2_{\gdbk}\left(l,u\right) \le \frac{\delta}{2}$.
To do that, we make use of the following \cref{lem.Hellinger2Gaussian}.

\begin{lem}[Proposition C.3 from \cite{maugis2011non}]\label{lem.Hellinger2Gaussian}
	Let $\Phi\left(\cdot;\mub_1,\Sigmab_1\right)$ and $\Phi\left(\cdot;\mub_2,\Sigmab_{2}\right)$ be two Gaussian densities with full rank covariance. It holds that
	\begin{align*}
		&d^2\left(\Phi\left(\cdot;\mub_1,\Sigmab_1\right),\Phi\left(\cdot;\mub_2,\Sigmab_{2}\right)\right) \\
		&= 2\left\{1-2^{D/2}\left|\Sigmab_1\Sigmab_{2}\right|^{-1/4}\left|\Sigmab_1^{-1}+\Sigmab_{2}^{-1}\right|^{-1/2}\exp\left[-\frac{1}{4}\left(\mub_1-\mub_2\right)^\top \left(\Sigmab_1+\Sigmab_{2}\right)^{-1}\left(\mub_1-\mub_2\right)\right]\right\}.
	\end{align*}
\end{lem}
Therefore, using the fact that $\cosh(t) = \frac{e^{-t}+e^{t}}{2}$, \cref{lem.Hellinger2Gaussian} leads to, for all $\yv \in \cY$:
\begin{align*}
	\hel(l(\cdot,\yv),u(\cdot,\yv)) 
	&
	= \int_{\cX}\left[ l(\xv,\yv) + u(\xv,\yv) - 2\sqrt{l(\xv,\yv)u(\xv,\yv)} \right]d\xv \\
	%
	%
	&= \left(1+2\alpha\right)^{-D}+ \left(1+2\alpha\right)^{D}-2\\
	&+\hel\left(\Phi\left(\cdot;\widetilde{\upsilonb}_{k,d}(\yv),\left(1+\alpha\right)^{-1}\widetilde{\Sigmab}_k\left(\bfB_k\right)\right),\Phi\left(\cdot;\widetilde{\upsilonb}_{k,d}(\yv),\left(1+\alpha\right)\widetilde{\Sigmab}_k\left(\bfB_k\right)\right)\right)\\
	&= 2\cosh\left[D\ln\left(1+2\alpha\right)\right]-2\\
	&\quad+2\left[1-2^{D/2}\left[\left(1+\alpha\right)^{-1}+\left(1+\alpha\right)\right]^{-D/2}\left|\widetilde{\Sigmab}_k\left(\bfB_k\right)\right|^{-1/2}\left|\widetilde{\Sigmab}_k\left(\bfB_k\right)\right|^{1/2}\right] 
	\\
	&= 2\cosh\left[D\ln\left(1+2\alpha\right)\right]-2+2-2\left[\cosh\left(\ln\left(1+\alpha\right)\right)\right]^{-D/2}\\
	& = 2g\left(D\ln\left(1+2\alpha\right)\right) + 2 h\left(\ln\left(1+\alpha\right)\right),
\end{align*}
where $g(t)=\cosh(t)-1 = \frac{e^{-t}+e^{t}}{2}  - 1$, and $h(t) = 1- \cosh(t)^{-D/2}$.
The upper bounds of terms $g$ and $h$ separately imply that, for all $\yv \in \cY$,
\begin{align*}
	\hel(l(\cdot,\yv),u(\cdot,\yv)) \le 2\left(2\cosh\left(\frac{1}{\sqrt{6}}\right)\alpha^2D^2 + \frac{1}{4}\alpha^2D^2\right) \le 6 \alpha^2 D^2 = \frac{\delta^2}{4},
\end{align*} 
where we choose $\alpha = \frac{3 \epsilon}{\lambda_m}, \epsilon = \frac{\delta \lambda_m}{6\sqrt{6}D}$, $\forall \delta \in (0,1],D \in \Ns, \lambda_m > 0$, which appears in \eqref{eq.conditionRadiusDiscVariance} and satisfies $\alpha=  \frac{\delta}{2\sqrt{6}D}$ and $0 < \epsilon < \frac{\lambda_m}{3}$. Indeed, studying functions $g$ and $h$ yields
\begin{align*}
	\bfg'(t) &= \sinh(t), \bfg''(t) = \cosh(t)\le  \cosh(c), \forall t \in [0,c], c \in\R_+,
	\\
	h'(t) &=\frac{D}{2} \cosh(t)^{-D/2-1} \sinh(t),\\ 
	h''(t) &=\frac{D}{2} \left(-\frac{D}{2}-1\right) \cosh(t)^{-D/2-2} \sinh^2(t) +\frac{D}{2} \cosh(t)^{-D/2}\\
	&=\frac{D}{2}\left(1-\left(\frac{D}{2}+1\right)\left(\frac{\sinh(t)}{\cosh(t)}\right)^2\right)  \cosh(t)^{-D/2} \le \frac{D}{2},
\end{align*}
where we used the fact that $\cosh(t) \ge 1$. Then, since $g(0) = 0,\bfg'(0) =0,h(0) = 0,h'(0) = 0$, by applying Taylor's Theorem, it is true that 
\begin{align*}
	g(t) &= g(t) - g(0) - \bfg'(0)t =  R_{0,1}(t) \le \cosh(c)\frac{t^2}{2},\forall t \in [0,c],\\
	h(t) &= h(t) - h(0) - h'(0)t =  R_{0,1}(t) \le \frac{D}{2}\frac{t^2}{2} \le \frac{D^2}{2}\frac{t^2}{2},\forall t \ge 0.
\end{align*}
We wish to find an upper bound for $t = D \ln\left(1+2 \alpha\right)$, $D \in \Ns$, $\alpha = \frac{\delta}{2\sqrt{6}D}$, $\delta \in (0,1]$. Since $\ln$ is an increasing function, then we have 
$$t = D \ln\left(1+\frac{\delta}{\sqrt{6}D}\right) \le D \ln\left(1+\frac{1}{\sqrt{6}D}\right) \le D \frac{1}{\sqrt{6}D}= \frac{1}{\sqrt{6}}, \forall\delta \in (0,1],$$ since $\ln\left(1+\frac{1}{\sqrt{6}D}\right) \le \frac{1}{\sqrt{6}D}$, $\forall D \in \Ns$. Then, since $\ln\left(1+ 2\alpha\right) \le 2\alpha, \forall \alpha \ge 0$, 
\begin{align*}
	g\left(D \ln\left(1+2 \alpha\right)\right) &\le \cosh\left(\frac{1}{\sqrt{6}}\right)\frac{\left(D \ln\left(1+2 \alpha\right)\right)^2}{2} \le \cosh\left(\frac{1}{\sqrt{6}}\right) \frac{D^2}{2} 4 \alpha^2,\\
	h\left(\ln\left(1+\alpha\right)\right) &\le \frac{D^2}{2}\frac{\left(\ln\left(1+\alpha\right)\right)^2}{2} \le \frac{D^2\alpha^2}{4}.
\end{align*}


%
Note that the set of $\delta/2$-brackets $[l,u]$ over $\gdbk$ is totally defined by the parameter spaces ${\tilde{S}}^{\disc}_{\bfB_k}(\epsilon)$ and $G_{\Upsilondk}\left(\delta_{\Upsilondk}\right)$. This leads to an upper bound of the $\delta/2$-bracketing entropy of $\gdbk$ evaluated from an upper bound of the two set cardinalities. Hence, given any $\delta > 0$, by choosing $\epsilon = \frac{\delta \lambda_m}{6\sqrt{6}D}$, $\alpha = \frac{3 \epsilon}{\lambda_m}=\frac{\delta}{2\sqrt{6}D}$, and $\delta^2_{\Upsilondk} = D \alpha \epsilon=D\frac{\delta}{2\sqrt{6}D}\frac{\delta \lambda_m}{6\sqrt{6}D} = \frac{\delta^2\lambda_m}{72D}$, it holds that
\begin{align*}
	\cN_{[\cdot],d_{\gdbk} }\left(\frac{\delta}{2},\gdbk\right)
	%
	&\le \card\left({\tilde{S}}^{\disc}_{\bfB_k}(\epsilon)\right)\times \card\left(G_{\Upsilondk}\left(\delta_{\Upsilondk}\right)\right) \\
	&\le \left(\Bigg\lfloor \frac{2\lambda_M}{\epsilon}\Bigg\rfloor\frac{D\left(D-1\right)}{2 D_{\bfB_k}}\right)^{D_{\bfB_k}} \left(\frac{\exp\left(C_{\Upsilondk}\right)}{\delta_{\Upsilondk}}\right)^{\dim\left(\Upsilondk\right)} \left(\text{using \eqref{eq.upperCardCovariance} and \eqref{eq.cardMeanGaussianEx}}\right)\\
	&\le \left(\frac{2\lambda_M 6 \sqrt{6}D}{\delta \lambda_m}\frac{D\left(D-1\right)}{2 D_{\bfB_k}}\right)^{D_{\bfB_k}} \left(\frac{6 \sqrt{2D}\exp\left(C_{\Upsilondk}\right)}{\delta \sqrt{\lambda_m}}\right)^{\dim\left(\Upsilondk\right)}\\
	&= \left(\frac{6 \sqrt{6} \lambda_M D^2\left(D-1\right)}{ \lambda_m D_{\bfB_k}}\right)^{D_{\bfB_k}} \left(\frac{6 \sqrt{2D}\exp\left(C_{\Upsilondk}\right)}{ \sqrt{\lambda_m}}\right)^{\dim\left(\Upsilondk\right)}\left(\frac{1}{\delta}\right)^{D_{\bfB_k}+\dim\left(\Upsilondk\right)}.
\end{align*}
Finally, by definition of bracketing entropy in \eqref{eq_definition_BracketingEntropy}, we obtain
\begin{align*}
	\cH_{[\cdot],d_{\gdbk} }\left(\frac{\delta}{2},\gdbk\right)
	%
	&\le D_{\bfB_k} \ln \left(\frac{6 \sqrt{6} \lambda_M D^2\left(D-1\right)}{ \lambda_m D_{\bfB_k}}\right) + \dim\left(\Upsilondk\right) \ln\left(\frac{6 \sqrt{2D}\exp\left(C_{\Upsilondk}\right)}{ \sqrt{\lambda_m}}\right)\nn\\
	&\quad +\left(D_{\bfB_k}+\dim\left(\Upsilondk\right)\right)\ln\left(\frac{1}{\delta}\right)
	= \dim\left(\gdbk\right) \left(C_{\gdbk} + \ln \left(\frac{1}{\delta}\right)\right),
\end{align*}
where $\dim\left(\gdbk\right) = D_{\bfB_k}+\dim\left(\Upsilondk\right)$ and $$C_{\gdbk} = \frac{D_{\bfB_k} \ln \left(\frac{6 \sqrt{6} \lambda_M D^2\left(D-1\right)}{ \lambda_m D_{\bfB_k}}\right) + \dim\left(\Upsilondk\right) \ln\left(\frac{6 \sqrt{2D}\exp\left(C_{\Upsilondk}\right)}{ \sqrt{\lambda_m}}\right)}{\dim\left(\gdbk\right)}.$$


\bibliographystyle{apalike2}
\bibliography{NAMS-BLoMPE.bib}

\end{document}